\documentclass{conm-p-l}
\usepackage{amssymb,url}
\usepackage[all]{xy}
\newtheorem{theorem}{Theorem}[section]
\newtheorem{lemma}[theorem]{Lemma}
\newtheorem{corollary}[theorem]{Corollary}
\newtheorem{proposition}[theorem]{Proposition}
\newtheorem{conjecture}[theorem]{Conjecture}

\theoremstyle{definition}
\newtheorem{definition}[theorem]{Definition}

\theoremstyle{remark}
\newtheorem{remark}[theorem]{\bf Remark}

\numberwithin{equation}{section}

\newcommand{\C}{\mathbb C}
\newcommand{\PP}{\mathbb P}
\newcommand{\Q}{\mathbb Q}
\newcommand{\G}{\mathbb G}
\newcommand{\Z}{\mathbb Z}

\newcommand{\N}{\mathbb N}

\newcommand{\A}{\mathbb A}
\newcommand{\R}{\mathbb R}
\newcommand{\LL}{\mathbb L}

\DeclareMathOperator{\AH}{AH} 
\DeclareMathOperator{\alg}{alg}
\DeclareMathOperator{\gpast}{ast}
\DeclareMathOperator{\AST}{AST}
\DeclareMathOperator{\Aut}{Aut}
\DeclareMathOperator{\CH}{CH} 
\DeclareMathOperator{\gpDH}{DH}
\DeclareMathOperator{\gpDL}{DL}
\DeclareMathOperator{\End}{End}
\DeclareMathOperator{\et}{et}
\DeclareMathOperator{\Fr}{Fr}
\DeclareMathOperator{\Gal}{Gal}
\DeclareMathOperator{\GIso}{GIso}
\DeclareMathOperator{\GL}{GL}
\DeclareMathOperator{\SL}{SL}
\DeclareMathOperator{\GO}{GO}
\DeclareMathOperator{\GSp}{GSp}
\DeclareMathOperator{\gpH}{H}
\DeclareMathOperator{\id}{id}
\DeclareMathOperator{\Iso}{Iso}
\DeclareMathOperator{\Ker}{Ker}
\DeclareMathOperator{\gpL}{L}
\DeclareMathOperator{\gpO}{O}
\DeclareMathOperator{\MMT}{MMT}
\DeclareMathOperator{\MS}{MS}
\DeclareMathOperator{\MT}{MT}
\DeclareMathOperator{\PHS}{PHS}
\DeclareMathOperator{\Sp}{Sp}
\DeclareMathOperator{\ST}{ST}
\DeclareMathOperator{\gpU}{U}
\DeclareMathOperator{\Zar}{Zar}

\begin{document}

\title[Motivic Serre group and algebraic Sato-Tate group]
{{Motivic Serre group, algebraic Sato-Tate group and Sato-Tate conjecture}}

\author[Grzegorz Banaszak]{Grzegorz Banaszak}
\address{Department of Mathematics and Computer Science, Adam Mickiewicz University,
Pozna\'{n} 61614, Poland. Current address: Department of Mathematics, University of California, San Diego, La Jolla, CA 92093, USA}
\email{banaszak@amu.edu.pl}

\author[Kiran S. Kedlaya]{Kiran S. Kedlaya}
\address{Department of Mathematics, University of California, San Diego, La Jolla, CA 92093, USA}
\email{kedlaya@ucsd.edu}
\urladdr{http://kskedlaya.org}



\thanks{G. Banaszak was supported by the NCN 2013/09/B/ST1/04416 
(National Center for Science of Poland), Weizmann Institute and Hebrew University in June 2012, 
Warsaw University Oct. 2012--Feb. 2013 and University of California, San Diego Sept. 2014--June 2015. 
K. Kedlaya was supported by NSF grant DMS-1101343 and UC San Diego
(Stefan E. Warschawski professorship). Thanks to Jean-Pierre Serre for answering a 
question and suggesting a correction concerning $l$-adic representations, and to Pierre Deligne and Norbert Schappacher for 
answering questions concerning motives associated with modular forms}
\begin{abstract}
We make explicit Serre's generalization of the Sato-Tate conjecture for motives, by expressing the construction in terms of fiber functors from the motivic category of 
absolute Hodge cycles into a suitable category of Hodge structures of odd weight. This extends the case of abelian varietes, which we treated in a previous paper. That description was used by Fit\'e--Kedlaya--Rotger--Sutherland to classify Sato-Tate groups of abelian surfaces; the present description is used by Fit\'e--Kedlaya--Sutherland to make a similar classification for certain motives of weight 3. We also give conditions under which verification of the Sato-Tate conjecture reduces to the identity connected component of the corresponding Sato-Tate group.
\end{abstract}

\maketitle


\section{Introduction}

In \cite{Se94}, Serre gave a general approach, in terms of the motivic category for numerical equivalence, 
towards the question of equidistribution of Frobenius elements in families of $l$-adic representations; this approach puts such questions as the Chebotarev density theorem and the Sato-Tate conjecture in a common framework.
Serre revisited this topic in \cite{Se12}, making the description somewhat more explicit.
The purpose of this paper is to follow in this direction, expressing Serre's construction in terms of fiber functors from 
the motivic category of absolute Hodge cycles into a suitable category of Hodge structures of odd weight. This extends our previous paper \cite{BK}, in which we carried out this program for abelian varieties; this was motivated by the immediate application to the classification of Sato-Tate groups of abelian surfaces in \cite{FKRS12}. Similarly, 
the results of this paper are used in \cite{FKS} to carry out a similar classification for a special class of motives of weight 3, and are expected to find further use in similar classifications for other classes of motives of odd weight. (Some modifications are needed to handle cases of even weight, such as K3 surfaces.)
The organization of the paper is as follows.
\medskip

In chapter 2, we briefly recall some facts about Hodge structures and Mumford-Tate groups 
in a fashion suitable for our exposition. 
\medskip

In chapter 3, we extend the notion of twisted decomposable Lefschetz group, introduced in \cite{BK}, to 
Hodge structures with some extra endomorphism structure. The twisted decomposable Lefschetz group (Definition 
\ref{twisted decomposable Lefschetz group}) is the disjoint sum of Galois twists 
(Definition \ref{Galois twist of the Lefschetz group}) of the Lefschetz group.
\medskip

In chapter 4, we work with Hodge structures associated with families of $l$-adic representations and prove 
basic results concerning relations between group schemes $G_{l, K, 1}^{\alg}$ and $G_{l, K}^{\alg}.$
\medskip

In chapter 5, we state the algebraic Sato-Tate conjecture for families of $l$-adic 
representations associated with Hodge structures. We also restate the Sato-Tate conjecture in this case and prove some 
basic properties of the algebraic Sato-Tate group and the Sato-Tate group. In particular, under the algebraic Sato-Tate conjecture, 
we establish the isomorphism (Proposition \ref{connected components iso}) between the groups of connected 
components of the algebraic Sato-Tate and Sato-Tate groups. We also introduce Galois twists inside $G_{l, K, 1}^{\alg}$
(see Definition \ref{twisted form of GlK and GlK1}) and we explain the relation of these twists to Galois twists of
the corresponding Lefschetz group.
\medskip

In chapter 6, under some mild assumptions on the base field $K,$ we compute connected components of $G_{l, K, 1}^{\alg}$
(Theorem \ref{jK0K and Zar1 are isomorphisms}). Then, under the algebraic Sato-Tate conjecture, we make a corresponding 
computation of connected components of $AST_{K}$ and $ST_{K}$ (Theorem \ref{STK iff STK0}). As a consequence, we prove that the Sato-Tate conjecture holds with respect to $ST_{K}$ if and only if it holds with respect to the connected component of $ST_{K}$ 
(Theorem \ref{STK iff STK0}). 
\medskip
  
In chapter 7, we show how to compute Mumford-Tate and Hodge groups for powers of Hodge structures and similarly how to 
compute $G_{l, K, 1}^{\alg}$ and $G_{l, K}^{\alg}$ for powers of $l$-adic representations. We also observe that in some 
cases, the Mumford-Tate conjecture implies the algebraic Sato-Tate conjecture.   
\medskip

In chapter 8, we continue the discussion from chapter 7 of the relationship between the algebraic Sato-Tate conjecture 
and the Mumford-Tate conjecture. We establish conditions for the algebraic Sato-Tate conjecture to hold with the algebraic Sato-Tate 
group equal to the corresponding twisted decomposable Lefschetz group. 
\medskip

Chapters 9--11 give the application of chapters 2--8 to the case where the polarized Hodge structures
and associated $l$-adic representations come from motives in the motivic category of absolute Hodge cycles
introduced by Deligne \cite{D1}, \cite{DM}. All the assumptions on Hodge structures and associated $l$-adic representations 
we made in chapters 2--5 are satisfied in this case. 
\medskip

At the beginning of chapter 9, we recall some results concerning the category $\mathcal{M}_{K}$ of 
motives for absolute Hodge cycles. Next, for a motive $M$ of $\mathcal{M}_{K}$ we introduce the 
Artin motive $h^0(D)$ corresponding to $D := \End_{\mathcal{M}_{\overline{K}}} (\overline{M})$ and compute the 
motivic Galois group $G_{\mathcal{M}^{0}_{K} (D)}$ of the smallest Tannakian subcategory 
$\mathcal{M}^{0}_{K} (D)$ of $\mathcal{M}_{K}$ generated 
by $h^0(D).$ Also, let $\mathcal{M}_{K} (M)$ denote the smallest Tannakian subcategory of $\mathcal{M}_{K}$ 
generated by $M.$ 
From this point on in the paper, we work only with \emph{homogeneous motives},
i.e., motives which occur as factors of motives of the form $h^r (X) (m)$
for some smooth projective variety $X$ over $K$ and some
and $m \in \Z$. For $M$ a homogeneous motive, we consider the 
motivic Galois group $G_{\mathcal{M}_{K} (M)}$ and the motivic Serre group $G_{\mathcal{M}_{K} (M), 1}$ (Definition  
\ref{motivic Galois group, motivic Serre group}). We also define Galois twists (Definition
\ref{Definition of GMKA1tau}) inside the 
motivic Serre group and explain their relation to Galois twists of the corresponding 
Lefschetz group. The precise computation of $G_{\mathcal{M}^{0}_{K} (D)}$
allows us to write down the motivic Serre group as a disjoint union of these twists
(see \eqref{GMKA1 decomposition into GMKA1tau}).

\medskip

At the beginning of chapter 10, we find a sufficient condition (Theorem 
\ref{equality of conn comp for GM and GM1}) for the natural map
$\pi_{0}(G_{\mathcal{M}_{K} (M), 1}) \rightarrow \pi_{0}(G_{\mathcal{M}_{K} (M)})$
to be an isomorphism. Theorem~\ref{equality of conn comp for GM and GM1}
is the motivic analogue of Theorem~\ref{equality of conn comp 
for Glalg and Glalg1}. We then introduce the motivic Mumford-Tate conjecture and motivic Sato-Tate 
conjecture.

\medskip

We start chapter 11 by recalling the relationship between the motivic Mumford-Tate group with the corresponding 
Mumford-Tate group and the relation of the motivic Serre group with the Hodge group. Under Serre's conjecture that
$\MT (V, \psi) = \MMT_{K} (M)^{\circ}$, i.e., that the Mumford-Tate group is equal to the connected component of the motivic Mumford-Tate group, we define (Definition \ref{GMKA1 as AST group}) the algebraic Sato-Tate group. We collect the main properties of the algebraic Sato-Tate group in Theorem \ref{AST as expected extension of HA}. At the end of this chapter, under the assumption that
$\gpH(V, \psi) = C_{D} (\Iso_{(V, \psi)}),$ we show that the algebraic Sato-Tate group is the corresponding 
twisted decomposable Lefschetz group (Corollary \ref{The natural candidate for AST group example 1}). In addition, under the Mumford-Tate conjecture, we prove the algebraic Sato-Tate conjecture in this case (Corollary \ref{cor alg Sato-Tate for MT explained by endo}). We finish by proving, under an assumption on homotheties in the associated $l$-adic representations and 
under the algebraic Sato-Tate Conjecture for the base field, that the Sato-Tate conjecture holds with respect to $ST_{K}$ if and only if it holds with respect to the connected component of $ST_{K}$ (Theorem \ref{ST with resp. to STK the same as with resp. to STK0}). 
Theorem \ref{ST with resp. to STK the same as with resp. to STK0} may serve of use in proving cases of the Sato-Tate conjecture, by making it possible to avoid computations involving connected components of the Sato-Tate group.

\medskip
In conclusion, recall that for Absolute Hodge Cycles (AHC) motives (Definition \ref{AHC motives}),
Serre's conjecture $\MT (V, \psi) = \MMT_{K} (M)^{\circ}$ holds (Remark 
\ref{A Serre conjecture about connected component of motivic Mumford-Tate}). 
Hence the algebraic Sato-Tate group is defined (Definition \ref{GMKA1 as AST group}) unconditionally 
for AHC motives. All motives associated with abelian varieties are AHC motives (\cite[Theorem 2.11]{D1}). Moreover, if $\mathcal{M}^{{\rm{av}}}_{K}$ 
denotes the Tannakian subcategory of $\mathcal{M}_{K}$ generated by abelian varieties and Artin motives, then every motive in 
$\mathcal{M}^{{\rm{av}}}_{K}$ is an AHC motive (\cite[Theorem 6.25]{DM}). So the algebraic Sato-Tate group is defined unconditionally 
for motives in $\mathcal{M}^{{\rm{av}}}_{K}$ (cf. \cite{BK}). It is shown in \cite[Proposition 6.26]{DM} that the motives associated 
with curves, unirational varieties of dimension $\leq 3,$ Fermat hypersurfaces, and K3 surfaces belong to $\mathcal{M}^{{\rm{av}}}_{K}.$ In general, the Hodge conjecture implies that every Hodge cycle on a motive is an algebraic cycle, and Deligne showed that every algebraic cycle is an absolute Hodge cycle (\cite[Example 2.1]{D1}). Hence the Hodge conjecture 
implies that every motive is an AHC motive.


\section{Hodge structures and Mumford-Tate group}
Let $(V, \psi)$ be a rational, polarized, pure Hodge structure of weight $n.$
Hence $V$ is a vector space over $\Q$ and $\psi$ is a bilinear nondegenerate
$(-1)^n$ symmetric form $\psi\, :\, V \times V \rightarrow \Q(-n)$ such that
$V_{\C}$ has a pure Hodge structure of weight $n.$ 
Let $\PHS(\Q)$ denotes the category of
rational, polarized, pure Hodge structures. 
The category $\PHS(\Q)$ is abelian and semisimple \cite[Lemme 4.2.3, p.\ 44]{D2} 
(cf. \cite[Cor. 2.12, p.\ 40]{PS}).  
Let $D_h := D (V, \psi) := \End_{\PHS(\Q)} (V, \psi).$ 
In particular $D_h$ is a finite-dimensional semisimple algebra over $\Q.$  
If $(V, \psi)$ is a simple polarized Hodge structure, then $D_h$ is a division algebra.
By the definition of a Hodge structure,
\begin{equation}
V_{\C} := V \otimes_{\Q} {\C}= \bigoplus_{n=p+q} V^{p,q} 
\label{Hodge decomposition}
\end{equation}
where \,${\overline{V^{p,q}}}= V^{q,p}$ and 
the $V^{p,q}$ are equivariant with respect to the action of the endomorphism algebra 
$D_h \otimes_{\Q} \C.$ Put $\psi_{\C} := \psi \otimes_{\Q} \C.$ 
Recall that $\Q (-n)$ is a pure Hodge structure of weight $2 n$ such that 
$\C (-n) = \C (-n)^{n,n}.$ The polarization $\psi$ can be seen as the morphism 
$\psi\, :\, V \otimes_{\Q} V \rightarrow \Q(-n)$
of pure Hodge structures of weight $2 n$ so that the $\C$-bilinear form 
$\psi_{\C} \, : \, V_{\C} \times V_{\C} \rightarrow \C (-n)$
has the property that $\psi_{\C} (V^{p,q} \times V^{p^{\prime},q^{\prime}}) = 0$
if $p + p^{\prime} \not= n$ or $q + q^{\prime} \not= n.$

\begin{remark}
More generally, $(T, \varphi)$ is an integral, polarized, pure Hodge structure of weight $n$
if $T$ is a free abelian group and $\varphi\, :\, T \times T \rightarrow \Z (-n)$ 
is a nondegenerate $\Z$-bilinear map such that $(V, \psi)$ is a rational, polarized pure Hodge 
structure of weight $n,$ where $V := T \otimes_{\Z} \Q$ and 
$\psi := \varphi \otimes \Q.$ Let $\PHS(\Z)$ denote the category of
integral, polarized, pure Hodge structures. 
\end{remark}

\begin{remark}
A recent, simple approach to real Hodge and mixed Hodge structures can be found in 
\cite{BM1}, \cite{BM2}. 
\label{remark on BM work}\end{remark}

\begin{remark} 
The vector space $V$ defines a commutative group scheme, also denoted $V$ by abuse of notation, whose points with values in a unital commutative $\Q$-algebra $R$ are:
$$V(R) \, := \, V \otimes_{\Q} R.$$
If $d := \dim_{\Q}\, V$, then any choice of basis of the vector space $V$ gives an isomorphism
$V \, \cong \, \A^d$ of group schemes over $\Q.$ We will equip $V$ 
with the tautological action of the group scheme $\GL_V$.   
\label{V as a group scheme}\end{remark}
\medskip

We will be particularly interested in those elements $g \in \GL_{V},$ for which there exists an element $\chi(g) \in \G_{m,\Q}$ such that $\psi (gv, gw) = \chi (g) \psi (v, w)$ for all $v, w \in V$. The following formulas determine group subschemes of $\GL_V$ of special interest.
\begin{align}
\GIso_{(V, \psi)} &:= \{g \in \GL_{V}:\, \psi (gv, gw) = \chi (g) \psi (v, w) \,\,\,\, \forall \, v, w \in V\},
\label{def of GIso} \\
\Iso_{(V, \psi)} &:= \{g \in \GL_{V}:\,\, \psi (gv, gw) =  \psi (v, w) \,\,\,\,  \forall \, v, w \in V\}.
\label{def of Iso}
\end{align} 
There is also a map of group schemes
$$\chi \, :\, \GIso_{(V, \psi)} \, \rightarrow \, \G_{m,\Q}$$
$$ g \, \mapsto \, \chi (g),$$
which is a character of $\GIso_{(V,\psi)}$ such that $\Iso_{(V, \psi)} = {\rm{Ker}}\, \chi.$ 

\begin{remark}\label{the image of homothety via chi}
Observe that for every $\alpha \in \G_{m,\Q}$ by bilinearity of $\psi$ we have:
\begin{equation}
\psi (\alpha \, {{\rm Id}}_{V} v,\,  \alpha \, {{\rm Id}}_{V} w) = \psi (\alpha v, \, \alpha w) =
\alpha^2 \, \psi(v, \, w).
\nonumber
\end{equation}
Hence 
\begin{equation}
\alpha \, {{\rm Id}}_{V} \in 
\GIso_{(V, \psi)}
\quad {{\rm and}} \quad \chi (\alpha \, {{\rm Id}}_{V}) = \alpha^2 .
\label{chi of homothety}
\end{equation}
\end{remark}
\medskip

We also observe that:
\begin{align}
\GIso_{(V, \psi)} \,\, &= \,\,
\left\{
\begin{array}{lll}
\GO_{(V, \psi)}&\rm{if}&n \,\,\, \rm{even}\\
\GSp_{(V, \psi)}&\rm{if}&n \,\,\, \rm{odd;}\\
\end{array}\right.
\\
\Iso_{(V, \psi)} \,\, &= \,\,
\left\{
\begin{array}{lll}
\gpO_{(V, \psi)}&\rm{if}&n \,\,\, \rm{even}\\
\Sp_{(V, \psi)}&\rm{if}&n \,\,\, \rm{odd.}\\
\end{array}\right.
\end{align}

\begin{definition}
For any pure Hodge structure (not necessarily polarized)
define the cocharacter \cite[p. 42]{D1}
\begin{equation}
\mu_{\infty, V} \, :\, \mathbb{G}_{m}({\C})\rightarrow \GL(V_{\C})
\label{cocharacter}
\end{equation}
such that for any $z\in {\C}^{\times},$ the automorphism $\mu_{\infty, V}(z)$ acts as 
multiplication by $z^{-p}$ on $V^{p,q}$ for each $p+q = n.$ 
\end{definition}
\medskip

\noindent
Notice that the complex conjugate cocharacter is 
\begin{equation}
\overline{\mu_{\infty, V}} \, :\, \mathbb{G}_{m}({\C})\rightarrow 
\GL(V_{\C})
\label{conjugate cocharacter}
\end{equation}
such that for any $z\in {\C}^{\times},$ the automorphism $\overline{\mu_{\infty, V}}(z)$ acts as 
multiplication by $\bar{z}^{-q}$ on $V^{p,q}$ for each $p+q = n.$ 
Observe that for $v \in V^{p,q}$ and $w \in V^{n-p,n-q}$  we have:
\begin{align}
\psi_{\C} (\mu_{\infty, V} (z) v, \, \mu_{\infty, V} (z) w) &= 
\psi_{\C} (z^{-p} v, \, z^{p-n} w) = z^{-n} \psi_{\C} (v, \,  w),
\label{mu infty and psi}
\\
\psi_{\C} (\overline{\mu_{\infty, V}} (z) \, v, \,\, \overline{\mu_{\infty, V}} \, (z) w) &= 
\psi_{\C} ({\bar z}^{-q} \, v, \,\, {\bar z}^{q-n} \,w) = {\bar z}^{-n} \psi_{\C} (v, \,  w).
\label{overline mu infty and psi}
\end{align}
 Hence
\begin{equation}
\mu_{\infty, V}(\C^{\times}) \, \subset \GIso_{(V,\, \psi)} (\C).
\label{muInfty subset GIso}\end{equation}
Since $D_h$ commutes  with $\mu_{\infty, V}(\C^{\times})$ on $V_{\C}$ elementwise,
it is clear that:
\begin{equation}
\mu_{\infty, V}(\C^{\times}) \, \subset C_{D_h} \GIso_{(V,\, \psi)} (\C).
\label{muInfty subset CDh GIso}
\end{equation}

\noindent
Let ${\mathbb S} := R_{\C/ \R} \, \G_{m}.$
The product $\mu_{\infty, V} \overline{\mu_{\infty, V}}$ restricted to each 
$V^{p,q}\oplus V^{q,p}$ gives the homomorphism of real algebraic groups:
\begin{equation}
h_{\infty, V}\, :\, {\mathbb S} \rightarrow \GL_{V_{\R}}.
\label{homomorphism h}\end{equation}
It follows from \eqref{mu infty and psi}, \eqref{overline mu infty and psi} that 
there is the following commutative diagram:
$$
{\xymatrix{
1 \ar[r]^{} & \Iso_{(V_{\R}, \psi_{\R})}  \ar[r]^{}  \quad 
& \quad \GIso_{(V_{\R}, \psi_{\R})}  \ar[r]^{}  \quad & 
\quad \G_{m} \ar[r]^{} & 1\\ 
1 \ar[r]^{} & \gpU (1) \ar@<0.1ex>[u]_{h_{\infty, V}}  \ar[r]^{} \quad  & 
\quad {\mathbb S} \ar@<0.1ex>[u]_{h_{\infty, V}} \ar[r]^{}  \quad & 
\quad  \G_m \ar@<0.1ex>[u]_{- n} \ar[r]^{} & 1 \\
}
\label{diagram for h maps}}$$

\begin{definition}\label{Definition Mumford Tate} (Mumford-Tate and Hodge groups)
\newline

\begin{enumerate}
\item{}
The \emph{Mumford-Tate group} of $(V, \psi)$ is the smallest algebraic subgroup 
$\MT(V, \psi) \subset \GIso_{(V, \psi)}$ over $\Q$ such that $\MT(V, \psi)({\C})$
contains $\mu_{\infty, V}({\C}).$ 
\item{} The \emph{decomposable Hodge group} is $\gpDH(V, \psi) := \MT(V, \psi) \cap 
\Iso_{(V, \psi)}$.
\item{}
The \emph{Hodge group} $\gpH(V, \psi) := \gpDH(V, \psi)^{\circ}$ is the connected
component of the identity in $\gpDH(V,\, \psi).$
\end{enumerate}
\end{definition}

We can equivalently define the Mumford-Tate and Hodge groups as follows.

\begin{enumerate}
\item{}
The \emph{Mumford-Tate group} of $(V, \psi)$ is the smallest algebraic subgroup 
$\MT(V, \psi) \subset \GIso_{(V, \psi)}$ over $\Q$ such that $\MT(V, \psi)({\C})$
contains $h_{\infty, V}({\mathbb S}({\C})).$ 
\item{}
The \emph{Hodge group} of $(V, \psi)$ is the smallest algebraic subgroup 
$\gpH(V, \psi) \subset \Iso_{(V, \psi)}$ over $\Q$ such that $\gpH(V, \psi)({\C})$
contains $h_{\infty, V}(U (1) ({\C})).$  
\end{enumerate}

\begin{remark}
Note that 
$\MT(V, \psi)$ is a reductive subgroup of $\GIso_{(V,\, \psi)}$ \cite[Prop.\ 3.6]{D1},
\cite[Th.\ 2.19]{PS}. It follows by (\ref{muInfty subset CDh GIso}) that
\begin{equation}
\MT(V, \psi) \subset C_{D_h}(\GIso_{(V,\, \psi)}). 
\label{MT subset CDh GIso}
\end{equation}
Moreover $\gpH(V, \psi)\, \subset \Iso_{(V,\, \psi)}$, hence
\begin{equation}
\gpH(V, \psi) \subset C_{D_h}(\Iso_{(V,\, \psi)}).
\label{DH and CDh}\end{equation}
For additional background on Mumford-Tate groups, see the lecture notes of
Moonen \cite{Mo1, Mo2}.
\end{remark}
  
In our investigation of the algebraic Sato-Tate group for Hodge structures, we will need to
investigate not only $D_h$ but possibly also other subrings $D \subset \End_{\Q} (V)$ such that $D$ acts 
on $V \otimes \C$ preserving the Hodge decomposition: $D V^{p,q} \subset V^{p,q}$ for all 
$p+q = n.$ Such a  $D$ commutes  with $\mu_{\infty, V}(\C^{\times})$ on $V_{\C}$ elementwise, hence:
\begin{equation}
\mu_{\infty, V}(\C^{\times}) \, \subset C_{D} \GIso_{(V,\, \psi)} (\C).
\label{muInfty subset CD GIso}
\end{equation}
Hence it follows by (\ref{muInfty subset CD GIso}) that
\begin{equation}
\MT(V, \psi) \subset C_{D}(\GIso_{(V,\, \psi)}). 
\label{MT subset CD GIso}
\end{equation}
\begin{equation}
\gpH(V, \psi) \subset C_{D}(\Iso_{(V,\, \psi)}).
\label{DH and CD}\end{equation}

\begin{definition}\label{Definition Lefschetz} 
The algebraic group: 
\begin{equation}
\gpL(V, \psi, D) := C_D^{\circ}(\Iso_{(V,\, \psi)})
\label{Def of Lefschetz group}\end{equation}
is called the {\it Lefschetz group} of $(V, \psi)$ and the ring $D.$ 
\end{definition}

\begin{remark}
By \eqref{DH and CD} and the connectedness of $\gpH(V, \psi)$, we have
\begin{equation}
\gpH(V, \psi) \subset \gpL(V, \psi, D).
\label{Hodge subset of Lefschetz}\end{equation} 
In particular
\begin{equation}
\gpH(V, \psi) \subset \gpL(V, \psi, D_h).
\label{Hodge subset of Lefschetz for Dh}\end{equation}  
\end{remark}


\section{Twisted Lefschetz groups}
Let $D \subset \End_{\Q} (V)$ be a subring such that the action of $D$ of
$V \otimes \C$ preserves the Hodge decomposition, i.e. $D V^{p,q} \subset V^{p,q}$ for all $p, q.$ 
\medskip

\noindent
Fix a number field $F$ and an algebraic closure $\overline{F}.$  
In this paper, $K/F$ will denote any finite extension contained in $\overline{F}.$
We assume that the ring $D$ admits a continuous representation of the absolute Galois group $G_F$
of $F$ such that its restriction to $G_K$ is denoted: 
\begin{equation}
\rho_{e} \, : \, G_K \rightarrow \Aut_{\Q} (D).
\end{equation}
\medskip

\noindent
\begin{definition}
The fixed field of the kernel of $\rho_e$ will be denoted: 
\begin{equation}
K_{e} \,\, := \,\, {\overline{K}}^{{\Ker \rho_{e}}}.
\nonumber\end{equation}
\label{def. of Le}
\end{definition}

\begin{remark}
The extension $K_{e}/K$ is finite 
Galois and $G_{K_{e}} = \Ker \rho_{e}.$ The field $K_e$ depends on $K$; in particular, it is not invariant under base change along an arbitrary extension $L/K.$ However, it is obvious that $K_e$ will not change if we change base along an
extension $L/K$ such that $L \subset K_e.$  
\label{base change and Le}
\end{remark}

\begin{definition}
 For $\tau \in \Gal(K_{e}/K)$, define:
\begin{equation}
\gpDL_{K}^{\tau}(V,\, \psi, D) := \{g \in \Iso_{(V,\, \psi)}: \, g \beta g^{-1} = \rho_e(\tau) (\beta) \,\,\,
\forall \,\beta \in D\}.
\label{decomposable twisted Lefschetz for fixed element}\end{equation} 
Because $D$ is a finite-dimensional  $\Q$-vector space, $\gpDL_{K}^{\tau}(V, \psi, D)$ is a closed subscheme of $\Iso_{(V,\, \psi)}$ for each $\tau.$ 
\label{Galois twist of the Lefschetz group}
\end{definition}

\begin{definition}
Define the \emph{twisted decomposable algebraic Lefschetz group} for the triple 
$(V,\, \psi, D)$ to be the closed algebraic subgroup of $\Iso_{(V,\, \psi)}$ given by
\begin{equation}
\gpDL_{K}(V, \psi, D) \,\,\, : = \bigsqcup_{\tau \in \Gal(K_{e}/K)} \,\, \gpDL_{K}^{\tau}(V, \psi, D).
\label{decomposition into twisted Lefschetz for fixed elements}
\end{equation} 
\label{twisted decomposable Lefschetz group}
\end{definition}

\noindent
For any subextension $L/K$ of ${\overline F}/K$, we have
$\gpDL_{L}(V, \psi, D) \subseteq \gpDL_{K}(V, \psi, D)$ and $\gpDL_{L}^{\id}(V, \psi, D) = 
\gpDL_{K}^{\id}(V, \psi, D)$.
Hence:
\begin{equation}
\gpL(V, \psi, D) = \gpDL^{\id}_{K}(V, \psi, D)^{\circ} = \gpDL_{K}(V, \psi, D)^{\circ} =
\label{twisted Lefschetz}\end{equation}
\begin{equation}
= \gpDL^{\id}_{L}(V, \psi, D)^{\circ} = \gpDL_{L}(V, \psi, D)^{\circ}.
\nonumber\end{equation}
In particular,
\begin{equation}
\gpDL_{K_e}^{\id}(V, \psi, D) = \gpDL_{K_e}(V, \psi, D) = 
\gpDL_{\overline F}(V, \psi, D) = \gpDL_{{\overline F}}^{\id}(V, \psi, D),
\label{simple properties of decomposible Lefschetz 1}
\end{equation}
\begin{equation}
\gpL (V, \psi, D) = \gpDL_{K_e}(V, \psi, D)^{\circ}= \gpDL_{\overline F}(V, \psi, D)^{\circ}.
\end{equation}

\begin{theorem}\label{Decomposible Lefschetz group standard properties} 
The twisted decomposable Lefschetz groups have the following 
properties.
\begin{itemize} 
\item[1.] $\gpDL_{K}^{\tau} (V^s, \psi^s, M_{s} (D)) \cong \gpDL_{K}^{\tau} (V, \psi, D)$
for every $\tau \in \Gal(K_e /K).$
\item[2.] Let $(V_{i}, \psi_{i})$ be polarized Hodge structures and let $D_i$ be finite-dimensional $\Q$-algebras preserving the Hodge structures $V_i.$ 
Let $D_i$ admit a continuous $G_K$-action. Put $(V, \psi) := 
\bigoplus_{i=1}^t (V_{i}, \psi_{i})$ and $D := \prod_{i=1}^t \, D_i.$
Then $\gpDL_{K}^{\tau} (V, \psi, D) \cong \prod_{i=1}^t \, \gpDL_{K}^{\tau} (V_i, \psi_{i}, D_i)$
for every $\tau \in \Gal(K_e /K).$ 
\item[3.] Let $(V_{i}, \psi_{i})$ be polarized Hodge structures and let $D_i$ be finite-dimensional $\Q$-algebras preserving the Hodge structures $V_i.$ 
Let $D_i$ admit a continuous $G_K$-action. Put $(V, \psi) := 
\bigoplus_{i=1}^t (V_{i}^{s_i}, \psi_{i}^{s_i})$ and $D := \prod_{i=1}^t \, M_{s_i} (D_i).$
Then $\gpDL_{K}^{\tau} (V, \psi, D) \cong \prod_{i=1}^t \, \gpDL_{K}^{\tau} (V_{i}, \psi_{i}, D_i)$
for every $\tau \in \Gal(K_e /K).$ 
\end{itemize}
\end{theorem}
\begin{proof}
1. Let $\Delta$ be the homomorphism that maps 
$\Iso_{(V, \psi)}$ naturally into 
\begin{equation}
{{\rm diag}} (\Iso_{(V, \psi)}, \dots,
\Iso_{(V, \psi)}) \subseteq \Iso_{(V^s, \psi^s)}.
\end{equation}
Since $\Q \subseteq D$, we have $M_{s} (\Q) \subseteq M_{s} (D).$ Directly from the
definition of the twisted decomposable Lefschetz group, we get
$\gpDL_{K}^{\tau} (V^s, \psi^s, M_{s} (D)) \cong \Delta (\gpDL_{K}^{\tau} (V, \psi, D)) 
\cong \gpDL_{K}^{\tau} (V, \psi, D)$.

\noindent
2. The proof is very similar to the proof of 1, using the fact that
$\prod_{i=1}^s \, \Q \subset \prod_{i=1}^t \, D_i.$  
  
\noindent
3. This follows immediately from 1 and 2. 
\end{proof}

\begin{remark}\label{Lefschetz group standard properties} 
Theorem \ref{Decomposible Lefschetz group standard properties} 
remains true if we replace $\gpDL_{K}^{\tau} (V^{\prime}, \psi^{\prime}, D^{\prime})$ 
with $\gpDL_{K}^{\tau} (V^{\prime}, \psi^{\prime}, D^{\prime})^{\circ}$
for all polarized Hodge structures $V^{\prime}, \psi^{\prime}$ and corresponding rings 
$D^{\prime}$ with Galois actions that appear in the theorem. Since we have $\gpL (V^{\prime}, \psi^{\prime}, D^{\prime}) = 
\gpDL_{K}^{\id} (V^{\prime}, \psi^{\prime}, D^{\prime})^{\circ}$, the Lefschetz group
satisfies properties 1--3 of Theorem 
\ref{Decomposible Lefschetz group standard properties}.
\end{remark}

\begin{remark}\label{decomposible bigger twisted Lefschetz}
Observe that we have
\begin{equation}
\gpDL_K(V, \psi, D) := \{g \in \Iso_{(V, \psi)}: \, \exists\,  \tau \in G_K \, 
\forall \, \beta \, \in D \quad  
g \beta g^{-1} = \rho_{e} (\tau) (\beta) \, \}
\nonumber\end{equation} 
Changing quantifiers we get another group scheme
\begin{equation}
\widetilde{DL}_K(V, \psi, D) := 
\{g \in \Iso_{(V, \psi)}: \, \forall \, \beta \, \in D \, \exists \, \tau \in G_K  
\quad g \beta g^{-1} = \rho_{e} (\tau) (\beta) \, \}
\label{decomposable bigger twisted Lefschetz1}\end{equation} 
Observe that 
$\gpDL_K(V, \psi, D) \subseteq \widetilde{\gpDL}_K(V, \psi, D).$
\end{remark}

\begin{remark}
Observe that (\ref{DH and CD}) implies that
\begin{equation}
\gpH(V, \psi) \subseteq \gpDL^{\id}_{K}(V, \psi, D) \subseteq \gpDL_K(V, \psi, D).
\label{DHA and DLKA}
\end{equation}
\end{remark}
\medskip

\section{Hodge structures associated with $l$-adic representations}

Let $(V, \psi)$ be a rational pure polarized Hodge structure of weight $n \not= 0.$ 
Put $V_l \cong V \otimes_{\Q} \Q_l$ and 
$\psi_{l} := \psi \otimes_{\Q} \Q_l.$ Let 
$(V_l, \psi_{l}) := (V \otimes_{\Q} \Q_l, \psi \otimes_{\Q} \Q_l)$ and assume that 
the bilinear form $\psi_l\, :\, V_l \times V_l \rightarrow \Q_l(-n)$  is $G_K$-equivariant and 
the family of $l$-adic representations 
\begin{equation}
\rho_l \, :\, G_K \rightarrow \GIso (V_l, \psi_{l})
\end{equation}
is of Hodge-Tate type and strictly compatible in the sense of Serre.
We assume that outside of a finite 
set of primes of $\mathcal{O}_K,$ for each $v$ the complex absolute values of the eigenvalues of a Frobenius element 
at $v$ are $q_{v}^{\frac{n}{2}}.$  
\medskip

The form $\psi_l$ is $(-1)^n$-symmetric by the assumptions on the Hodge structure.
Hence
\begin{equation}
\GIso (V_l, \psi_{l}) \,\, = \,\,
\left\{
\begin{array}{lll}
\GO (V_l, \psi_{l})&\rm{if}&n \, \rm{even;}\\
\GSp (V_l, \psi_{l})&\rm{if}&n \, \rm{odd.}\\
\end{array}\right.
\end{equation}
Let $\chi$ be the character defined in (\ref{def of GIso}) and 
let $\chi_{c} \, : G_K \rightarrow \Z_{l}^{\times}$ be the cyclotomic character. Then by the $G_K$-equivariance of 
$\psi_l$ we obtain: 
\begin{equation}
\chi \circ \rho_{l}  \, = \, \chi_{c}^{-n}.
\label{comatibility of chi with chi-cycl.}
\end{equation}

\begin{remark}\label{properties of Hodge-Tate representations } 
For a representation $\rho_l$ of Hodge-Tate type, the theorem of Bogomolov on homotheties 
(cf.\ \cite[Prop.\ 2.8]{Su}) applies, meaning that $\rho_{l}(G_K) \, \cap\,  \Q_{l}^{\times} \, {\rm{Id}}_{V_l}$ is open in $\Q_{l}^{\times} \, {\rm{Id}}_{V_l}.$ 
Moreover, Bogomolov proved \cite[Th{\' e}or{\` e}me 1]{Bo} that $\rho_{l} (G_K)$ is open 
in $G_{l, K}^{\alg}(\Q_l).$
\end{remark}

\begin{remark}\label{Hodge-Tate representations in etale cohomology} 
Strictly compatible families of $l$-adic representations of Hodge-Tate type arise naturally from {\' e}tale cohomology.
Indeed, if $X/K$ is a proper scheme and ${\overline X} := X \otimes_{K} {\overline F}$ 
then $V_{l, \et}^{i} := H^{i}_{\et} ({\overline X}, \, \Q_l)$ is potentially semistable for each
$G_{K_v}$-representation for every $v | l$ (see \cite[Cor.\ 2.2.3]{Ts1}, \cite{Ts2}). 
Hence the representation 
\begin{equation}
\rho_{l, \et}^{i} \, :\, G_K \rightarrow \GL (V_{l, \et}^{i})
\label{etale cohom. representation}
\end{equation}
is of Hodge-Tate type (cf.\ \cite[p.\ 603]{Su}). 
\end{remark}
\medskip

\begin{definition} Let
\begin{equation}
G_{l, K}^{\alg} := G_{l, K}^{\alg}(V, \psi) \subset  \GIso_{(V_{l}, \psi_{l})}
\label{Def of GlKalg}\end{equation} 
be the Zariski closure of $\rho_{l} (G_K)$ in $\GIso_{(V_{l}, \psi_{l})}.$ Put:
\begin{equation}
\rho_{l} (G_K)_{1} := \rho_{l} (G_K) \cap \Iso_{(V_{l}, \psi_{l})},
\label{Def of rholGK1}\end{equation}
\begin{equation}
G_{l, K, 1}^{\alg} := G_{l, K, 1}^{\alg}(V, \psi) :=  G_{l, K}^{\alg} \cap
\Iso_{(V_{l}, \psi_{l})}.
\label{Def of GlK1alg}\end{equation}
\end{definition}
\bigskip

\noindent
By the theorem of Bogomolov on homotheties 
(see Remark 
\ref{properties of Hodge-Tate representations }),  there is an exact sequence
\begin{equation}
1 \,\, {\stackrel{}{\longrightarrow}} \,\, G_{l, K, 1}^{\alg} \,\, {\stackrel{}{\longrightarrow}} \,\,  G_{l, K}^{\alg} \,\, {\stackrel{\chi}{\longrightarrow}} \,\, \G_m \,\, {\stackrel{}{\longrightarrow}} \,\, 1.
\label{The exact sequence for GlK1alg and GlKalg}
\end{equation}

\begin{remark}\label{if rhol semisimple then GlK1alg reductive}
If $\rho_l$ is semisimple, then $G_{l, K}^{\alg}$ is reductive; hence in this case the algebraic group $G_{l, K, 1}^{\alg}$ is also reductive, by virtue of being the kernel of a homomorphism from a reductive group to a torus. 
\end{remark}
\medskip

Naturally $\rho_{l} (G_K)_{1} \subseteq G_{l, K, 1}^{\alg}$. 
Let $K \subseteq L \subset \overline{F}$ be a tower of extensions with
$L/K$ finite.
Consider the following commutative diagram with left and middle vertical arrows injective:

\begin{equation}
{\xymatrix{
1 \ar[r]^{} & \,\, G_{l, K, 1}^{\alg}  \ar[r]^{}  
& \quad G_{l, K}^{\alg}  \ar[r]^{\chi}  \quad & \,\,
\G_{m} \ar[r]^{} & 1\\ 
1 \ar[r]^{} & \,\, G_{l, L, 1}^{\alg} \ar@<0.1ex>[u]_{}  \ar[r]^{}  & 
\quad  G_{l, L}^{\alg} \ar@<0.1ex>[u]_{}  \ar[r]^{\chi}  \quad & 
\,\, \G_m \ar@<0.1ex>[u]_{=} \ar[r]^{} & 1 \\
}\label{diagram GlL1 maps to GlK1 and GlL maps to GlK}}
\end{equation}

It is clear that $G_{l, K, 1}^{\alg} \, \cap \,G_{l, L}^{\alg} \, = \,  G_{l, L, 1}^{\alg}.$
If $L/K$ is Galois, then it follows from the diagram \eqref{diagram GlL1 maps to GlK1 and GlL maps to GlK} that there is 
a monomorphism:

\begin{equation}
j_{L/K} \, : \,\,\,\, \rho_{l} (G_K)_1 \, / \, \rho_{l} (G_L)_1 \,\,\, \hookrightarrow \,\,\, \rho_{l} (G_K) \, /\,  \rho_{l} (G_L).
\label{diagram rho GLl1 maps to rho GKll and rho GLl maps to rho GK}\end{equation}
\medskip

\begin{proposition}
Let $K \subset L \subset M$ with $M/K$ and $L/K$ Galois. The map $j_{M/K}$ is an isomorphism
if and only if $j_{M/L}$ and $j_{L/K}$ are isomorphisms.
\label{M over K gives iso iff L over K and M over L give iso}\end{proposition} 

\noindent
We observe that for any finite Galois extension $L/K$ the natural map is an epimorphism $\Zar_{L/K} := \Zar_{l, \, L/K}$:
\begin{equation}
{\xymatrix{
\Zar_{L/K} \, : \,\,\,\, \rho_{l} (G_K) \, /\,  \rho_{l} (G_L) \,\, \ar@{>>}[r]^{} & \quad G_{l, K}^{\alg} \, / \, G_{l, L}^{\alg}}}.
\label{GK mod GL surjects to GKalg mod GLalg}
\end{equation}

The proofs of the following three results:
Theorem \ref{pre Serre theorem}, Proposition \ref{L0realizing conn comp}, Theorem \ref{equality of conn comp for Glalg and Glalg1}, are similar to the proofs of  
\cite[Theorem 3.1, Proposition 3.2, Theorem 3.3]{BK}.
Theorem \ref{equality of conn comp for Glalg and Glalg1} is a generalization of the result of Serre \cite[\S 8.3.4]{Se12}. As usual, for an algebraic group $G$ we put $\pi_{0} (G) := G / G^{\circ}.$  

\begin{theorem}
Let $K \subseteq L \subset \overline{F}$ with
$L/K$ finite Galois. The following natural map is an isomorphism of finite groups:
\begin{equation}
i_{L/K} \,\, :\,\, G_{l, K, 1}^{\alg} /G_{l, L, 1}^{\alg}
\quad {\stackrel{\cong}{\longrightarrow}} \quad
G_{l, K}^{\alg} /G_{l, L}^{\alg}.
\label{iL is an isomorphism}
\end{equation}
In particular there are the following equalities:
\begin{equation}
(G_{l,L}^{\alg})^\circ = (G_{l,K}^{\alg})^\circ \quad {\rm{and}}\quad 
(G_{l,L,1}^{\alg})^\circ = (G_{l,K,1}^{\alg})^\circ.
\label{GlalgL0 is GlalgK0}\end{equation}
\label{pre Serre theorem}\end{theorem}
\begin{proof}
It is clear that $G_{l, L}^{\alg} \,\, \triangleleft \,\, G_{l, K}^{\alg}$ and $G_{l, L, 1}^{\alg}  \,\, \triangleleft \,\, G_{l, K, 1}^{\alg}.$ 
On the other hand, there is a surjective homomorphism
$\rho_l(G_K)/\rho_l(G_{L}) \to G_{l,K}^{\alg}/G_{l,L}^{\alg}$, so $G_{l,L}^{\alg}$ is a 
subgroup of $G_{l,K}^{\alg}$
of finite index. In particular, $(G_{l,L}^{\alg})^\circ = (G_{l,K}^{\alg})^\circ$. 
\medskip

\noindent
The following commutative diagram has exact rows. The left and the middle 
columns are also exact cf. (\ref{The exact sequence for GlK1alg and GlKalg}). 

$$
\xymatrix{
\quad & \quad 1 \ar@<0.1ex>[d]^{}  \quad & \quad 1 \ar@<0.1ex>[d]^{} \quad & 
1 \ar@<0.1ex>[d]^{} \\ 
1 \ar[r]^{} \quad & \quad G_{l, L, 1}^{\alg} \ar@<0.1ex>[d]^{} \ar[r]^{}  \quad 
& \quad  G_{l, K, 1}^{\alg} \ar@<0.1ex>[d]^{} \ar[r]^{}  \quad 
& \quad G_{l, K, 1}^{\alg} /G_{l, L, 1}^{\alg}  
\ar@<0.1ex>[d]^{i_{L}}_{\cong} \ar[r]^{}  \quad 
& \,\, 1 \\ 
1 \ar[r]^{} \quad & \quad G_{l, L}^{\alg} \ar@<0.1ex>[d]^{\chi} \ar[r]^{}  \quad 
& \quad  G_{l, K}^{\alg} \ar@<0.1ex>[d]^{\chi} \ar[r]^{}  \quad 
& \quad G_{l, K}^{\alg} /G_{l, L}^{\alg}   \ar@<0.1ex>[d]^{} \ar[r]^{}  \quad 
&  \,\, 1 \\
1 \ar[r]^{} \quad & \quad \G_m  \ar@<0.1ex>[d]^{}  \ar[r]^{=}  
\quad & \quad \G_m  \ar@<0.1ex>[d]^{}  \ar[r]^{} \quad & 1\\
 \quad & \quad 1   \quad & 1 \\}
\label{diagram before the Serre theorem}$$
Then a diagram chase (as in the snake lemma) shows that the third column is also 
exact, so the map $i_L$ is an isomorphism. 
Hence it is clear that $(G_{l,L,1}^{\alg})^\circ = (G_{l,K,1}^{\alg})^\circ$.
\end{proof}

\noindent
\begin{proposition} \label{L0realizing conn comp} Let the weight of the Hodge structure be the odd integer $n = 2m + 1$. There is a finite Galois extension $L_{0}/K$ 
such that $G_{l, L_{0}}^{\alg} = (G_{l, K}^{\alg})^{\circ}$ and $G_{l, L_{0}, 1}^{\alg} = ({G_{l, K,
1}^{\alg}})^{\circ}.$ 
\end{proposition}
\begin{proof}
Since the subscheme $({G_{l, K}^{\alg}})^{\circ}$ is open and closed in 
$G_{l, K}^{\alg}$ and $\rho_l$ is continuous, 
we can find a finite Galois extension $L_0/K$ such that
$\rho_l(G_{L_0}) \subset (G_{l, K}^{\alg})^{\circ}$. Hence $G_{l, L_{0}}^{\alg} \subseteq ({G_{l, K}^{\alg}})^{\circ}.$ Since  we already have the reverse inclusion, we obtain the first desired equality. 

Consider the restriction of the $l$-adic representation to the base field $L_0.$
Using the Hodge-Tate property of $V_l$, after taking $\C$ points in the exact sequence  
(\ref{The exact sequence for GlK1alg and GlKalg}) one can apply the homomorphism $h$ 
\cite[p.\ 114]{Se12} defined by Serre to get the homomorphism:
\begin{equation}
h : \G_m (\C) \rightarrow  G_{l, L_{0}}^{\alg} (\C)
\nonumber\end{equation}
such that for all $x \in \G_m(\C)$, $h(x)$ acts by multiplication by $x^p$ on the subspace $V^{p, n-p}$. One checks that $\chi (h(x)) = x^n$ for every $x \in \G_m (\C)$ (see the
diagram preceding Definition~\ref{Definition Mumford Tate}).
Let 
\begin{equation}
w : \G_m (\C) \rightarrow  G_{l, L_{0}}^{\alg} (\C)\\
\nonumber\end{equation}
$$w(x) = x \, \rm{Id}_{V_{\C}}$$ 
be the diagonal homomorphism; this is well-defined thanks to
Remark~\ref{properties of Hodge-Tate representations }.
We know (Remark \ref{the image of homothety via chi}) that $\chi (w(x)) = x^2$ for every 
$x \in \G_m (\C).$ Hence the homomorphism
\begin{equation}
s : \G_m (\C) \rightarrow  G_{l, L_{0}}^{\alg} (\C)
\nonumber
\end{equation}
$$s(x) := h(x) w(x)^{-m}$$ 
is a splitting of $\chi$ in the following exact sequence:
  
\begin{equation}
1 \,\, {\stackrel{}{\longrightarrow}} \,\, G_{l, L_{0}, 1}^{\alg} (\C) \,\, {\stackrel{}{\longrightarrow}} \,\,  G_{l, L_{0}}^{\alg} (\C) \,\, {\stackrel{\chi}{\longrightarrow}} \,\, \C^{\times} \,\, {\stackrel{}{\longrightarrow}} \,\, 1.
\nonumber\end{equation}
Observe that $G_{l, L_{0}}^{\alg} (\C)$ is a connected Lie group. Take any two points 
$g_0$ and $g_1$ in $G_{l, L_{0}, 1}^{\alg} (\C).$ There is a path $\alpha (t) \in 
G_{l, L_{0}}^{\alg} (\C)$ connecting $g_0$ and $g_1$, i.e., $\alpha (0) = g_0$ and
$\alpha (1) = g_1.$ Define a new path 
$$\beta (t) := {s (\chi (\alpha (t)))}^{-1} \alpha (t) \in G_{l, K_{0}}^{\alg} (\C)$$
Observe that: $\chi (\beta (t)) := \chi ({s (\chi (\alpha (t)))}^{-1}) \chi(\alpha (t)) = 
\chi (\alpha (t))^{-1} \chi(\alpha (t)) = 1.$ We easily check that
$\beta (0) = g_0$ and $\beta (1) = g_1.$ Hence $\beta (t) \in G_{l, K_{0}, 1}^{\alg} (\C)$ connects
$g_0$ and $g_1.$ It follows that $G_{l, L_{0}, 1}^{\alg}$ is connected, hence 
$G_{l, L_{0}, 1}^{\alg} = (G_{l, K, 1}^{\alg})^{\circ}.$ \end{proof}

\noindent
\begin{theorem}\label{equality of conn comp for Glalg and Glalg1} 
Let $n$ be odd. The following natural map is an isomorphism:
$$i_{CC} \,\, : \,\,  \pi_{0} (G_{l, K, 1}^{\alg}) \,\,\, 
{\stackrel{\cong}{\longrightarrow}} \,\,\, \pi_{0} (G_{l, K}^{\alg}).$$ 
\end{theorem}
\begin{proof}
Choose $L_0$ as in Proposition~\ref{L0realizing conn comp}. Put $L := L_0$ in 
the diagram of the proof of Theorem \ref{pre Serre theorem}. Then 
$i_{CC} = i_{L_0}$, which is an isomorphism by Theorem \ref{pre Serre theorem}.
\end{proof}

\medskip

\begin{remark}
The natural continuous action by left translation:
\begin{equation}
G_K \times \, \pi_{0} ( G_{l, K}^{\alg} ) \rightarrow \pi_{0} (G_{l, K}^{\alg}) 
\label{ContLeftTransOnConnecComp}\end{equation}
and Theorem \ref{equality of conn comp for Glalg and Glalg1} give the following continuous action by left translation:
\begin{equation}
G_K \times \, \pi_{0}(G_{l, K,1}^{\alg}) \rightarrow \pi_{0}(G_{l, K,1}^{\alg}).
\end{equation}
\end{remark}


\section{Algebraic Sato-Tate conjecture}
In this chapter we assume that the Hodge structure $(V, \psi),$ the ring $D$ and 
the family of $l$-adic representations $\rho_l\, :\, G_K \rightarrow \GIso (V_l, \psi_l)$ 
satisfy all the properties assumed in chapters 2--4.  We also assume hereafter that $n$ is odd; the case where $n$ is even requires some modifications to the definitions, which we will discuss elsewhere.

One of the main objectives of this paper is the investigation of the following conjecture:

\begin{conjecture}\label{general algebraic Sato Tate conj.} 
(Algebraic Sato-Tate conjecture) \newline
${\rm{(a)}}$ For every finite extension $K/F$ and for every $l$, there exist a natural-in-$K$ 
reductive algebraic group $\AST_{K} (V, \psi) 
\subset \Iso_{(V, \psi)}$ over $\Q$ and a natural-in-$K$ monomorphism of group schemes:
\begin{equation}
\gpast_{l, K} \,\, : \,\,  G_{l, K, 1}^{\alg} \,\, 
{\stackrel{}{\hookrightarrow}} \,\,  \AST_{K} (V, \psi)_{\Q_l}.
\label{Algebraic Sato-Tate monomorphism}\end{equation}

\noindent
${\rm{(b)}}$ The map
(\ref{Algebraic Sato-Tate monomorphism}) is an isomorphism: 
\begin{equation}
\gpast_{l, K} \,\, : \,\,  G_{l, K, 1}^{\alg} \,\, 
{\stackrel{\cong}{\longrightarrow}} \,\,  \AST_{K} (V, \psi)_{\Q_l}.
\label{Algebraic Sato-Tate equality generalized}\end{equation} 
\end{conjecture}
\noindent

\begin{remark} \label{reductive algebraic group}
We say that an algebraic group is \emph{reductive} if its identity connected component is reductive.
\end{remark}

\begin{remark}
The requirement that $\AST_K(V, \psi)$ and (\ref{Algebraic Sato-Tate monomorphism})
are natural in $K$ means that for any finite extension $L/K$ 
there is a natural monomorphism of group schemes: 
\begin{equation}
\AST_{L} (V, \psi) \hookrightarrow \AST_{K} (V, \psi)
\nonumber\end{equation} 
making the following diagram commute:

$${\xymatrix{
\quad G_{l, K, 1}^{\alg} \ar[r]^{\gpast_{l, K}} \quad 
& \quad  \AST_{K}(V, \psi)_{\Q_l}  \\ 
\quad G_{l, L, 1}^{\alg} \ar[r]^{\gpast_{l, L}} \ar@<0.1ex>[u]^{}  \quad 
& \quad   \AST_{L} (V, \psi)_{\Q_l} \ar@<0.1ex>[u]^{}  \\
}}
\label{initial diagram}$$
\label{naturality of ast l K}\end{remark}

\noindent
\begin{definition} The group $\AST_{K}(V, \psi)$ is called the \emph{algebraic Sato-Tate group}.
A maximal compact subgroup of $\AST_{K}(V, \psi)(\C)$ is called the \emph{Sato-Tate group} and is denoted 
$\ST_{K}(V, \psi).$ 
\label{The sato-Tate group}
\end{definition}

\noindent
\begin{remark} \label{abbreviation for ASTKVpsi}
We will make the following abbreviations: $\AST_{K} := \AST_{K} (V, \psi)$ and 
$\ST_{K} := \ST_{K} (V, \psi)$, whenever they do not lead to a notation conflict.
\end{remark}

\noindent
\begin{remark}\label{relation to the Tate conjecture}
When the Hodge structure $(V, \psi)$ comes from the cohomology of a smooth, projective variety
over $K$, then Conjecture \ref{general algebraic Sato Tate conj.} is closely related 
to the Tate conjecture. 
\end{remark}

Choose a suitable field embedding $\Q_l \rightarrow \C$ and
put  $ {G_{l, K, 1}^{\alg}}_{\C} := G_{l, K, 1}^{\alg} \otimes_{\Q_l} \C.$ 
Naturally we have $\pi_{0} (G_{l, K, 1}^{\alg}) \, \cong \, \pi_{0}({G_{l, K, 1}^{\alg}}_{\C}).$
By Theorem \ref{equality of conn comp for Glalg and Glalg1} and an 
argument similar to the proof of \cite[Lemma~2.8]{FKRS12}, we have the following.

\begin{proposition}\label{connected components iso}
Assume that the algebraic Sato-Tate conjecture (Conjecture~\ref{general algebraic Sato Tate conj.})
holds. Then there are natural isomorphisms
\begin{equation}
\pi_{0} (G_{l, K, 1}^{\alg}) \,\,\, \cong \,\,\, 
\pi_{0}(\AST_{K} (V, \psi)) \,\,\, \cong \,\,\, 
\pi_{0}^{}(\ST_{K}(V, \psi)).
\label{ConnCompIsom}\end{equation}
\end{proposition} 

\begin{remark} Assume that the algebraic Sato-Tate conjecture (Conjecture \ref{general algebraic Sato Tate conj.}) holds. 
Then obviously the Sato-Tate group $\ST_{K}(V, \psi)$ is independent of $l.$ Take a prime $v$ in $\mathcal{O}_K$
and take a Frobenius element $\Fr_{v}$ in $G_K.$
Following \cite[\S 8.3.3]{Se12} (cf.  \cite[Def.\ 2.9]{FKRS12}) one can make the following construction. Let $s_v$ be the semisimple part in $\SL_V(\C)$ of the element
$$q_{v}^{-\frac{n}{2}} \rho_{l} (\Fr_{v}) \,\, \in \,\, 
 G_{l, K, 1}^{\alg} (\C) \,\, \cong \,\, \AST_{K} (V, \psi) (\C) \,\, \subset 
 \,\, \Iso_{(V, \psi)} (\C) \,\, \subset \,\, \SL_{V} (\C);$$ 
since the family $(\rho_{l})$ is 
strictly compatible, $s_v$ is independent of $l$. By \cite[Theorem 15.3 (c) p.\ 99]{Hu}, the semisimple part
of $q_{v}^{-\frac{n}{2}} \rho_{l} (\Fr_{v})$ considered in 
$\Iso_{(V, \psi)} (\C)$ and in $\AST_{K} (V, \psi) (\C)$ is again $s_v,$ and so is again 
independent of $l.$ Hence ${\rm{conj}}(s_v)$ in $\AST_{K} (V, \psi) (\C)$ is independent of
$l.$ Obviously ${\rm{conj}}(s_v) \subset \AST_{K} (V, \psi) (\C)$ is independent of the 
choice of a Frobenius element $\Fr_{v}$ over $v$ and contains the
semisimple parts of all the elements of ${\rm{conj}}(q_{v}^{-\frac{n}{2}} \rho_{l} (\Fr_{v}))$
in $\AST_{K} (V, \psi) (\C).$ Moreover, the elements in ${\rm{conj}} (s_v)$ have eigenvalues of 
complex absolute value $1$ by our assumptions, so there is some conjugate of $s_v$ contained in $\ST_{K}(V, \psi).$ This allows us to make sense of the following conjecture.
\label{Sato-tate set up}\end{remark}

\begin{conjecture}\label{general Sato Tate conj.} (Sato-Tate conjecture) 
The conjugacy classes ${\rm{conj}}(s_v)$ in $\ST_{K}(V, \psi)$ are equidistributed in 
${\rm{conj}} (\ST_{K}(V, \psi))$ with respect to the measure induced by the 
Haar measure of $\ST_{K}(V, \psi).$   
\end{conjecture}
\bigskip

\begin{remark}
If we are only interested in the isomorphism (\ref{ConnCompIsom}) for a fixed $l$, then it is enough to 
assume the existence of $\AST_{K} (V, \psi),$ as in Conjecture~\ref{general algebraic Sato Tate conj.}, and the existence of $\gpast_{l, K}$ which is an isomorphism for this particular $l.$   
\label{connected components iso for given l} 
\end{remark}

We now use the twisted Lefschetz group to obtain an upper bound on the algebraic Sato-Tate group.
\begin{remark}
%
We will assume in this and the next three chapters that the induced action of $D$ on $V_l$ is $G_K$-equivariant. 
In other words, $\forall \beta \in D,$ $\forall v_l \in V_l$ and $\forall \sigma \in G_K:$
\begin{equation}
\rho_{l} (\sigma) (\beta \, v_l) = \sigma(\beta) \, \rho_{l} (\sigma) (v_l).
\end{equation}
This immediately gives:
\begin{equation}
\rho_{l} (\sigma) \beta \rho_{l} (\sigma^{-1}) (v_l) = \sigma (\beta) (v_l),
\end{equation}
\begin{equation}
\rho_{l} (G_K)_1 \subseteq \gpDL_K(V, \psi, D) (\Q_l).
\end{equation}
\end{remark}

\noindent
We will observe, by (\ref{GlK1algtau subset DLKtauVpsiDQl}) and 
(\ref{decomposition of Glk1 into twists}) below, that:
\begin{equation}
G_{l, K, 1}^{\alg} \subseteq \gpDL_K(V, \psi, D)_{\Q_l},
\end{equation}
\medskip

We are interested in finding polarized Hodge structures $(V, \psi)$ and rings 
$D$ for which $G_{l, K, 1}^{\alg} = \gpDL_K(V, \psi, D)_{\Q_l}$ for each $l$.  
In such cases $\AST_{K}(V, \psi) = \gpDL_K(V, \psi, D).$
We explain in this chapter that the equality $G_{l, K, 1}^{\alg} = \gpDL_K(V, \psi, D)_{\Q_l}$
is equivalent to $G_{l, K_{e}, 1}^{\alg} = \gpDL_{K_{e}} (V, \psi, D)_{\Q_l}$.

\begin{definition} Put: 
$$(G_{l, K}^{\alg})^{\tau} := \{g \in G_{l, K}^{\alg}: \, g \beta g^{-1} = \rho_e(\tau) (\beta) \,\,\,\,
\forall \,\beta \in D\},$$
$$(G_{l, K, 1}^{\alg})^{\tau} := (G_{l, K}^{\alg})^{\tau} \cap G_{l, K, 1}^{\alg}.$$
\label{twisted form of GlK and GlK1}\end{definition}

\noindent
Observe that
\begin{equation}
(G_{l, K, 1}^{\alg})^{\tau}  \subseteq \gpDL_{K}^{\tau}(V, \psi, D)_{\Q_l}.
\label{GlK1algtau subset DLKtauVpsiDQl}\end{equation}
\medskip

\begin{remark}\label{definitions of twisted GKalg and GK1alg}
Let $\tilde\tau \in G_K$ be a lift of $\tau \in \Gal(K_{e}/K).$ The coset
$\tilde\tau \, G_{K_{e}}$ does not depend on the lift. The Zariski closure of  
$\rho_{l} (\tilde\tau \, G_{K_{e}}) = \rho_{l} (\tilde\tau) \, \rho_{l} (G_{K_{e}})$ 
in $\GIso_{(V_l, \psi_l)}$ is  $\rho_{l} (\tilde\tau) \, G_{l, K_{e}}^{\alg}.$  
Since $\rho_{l} (\tilde\tau) \, \rho_{l} (G_{K_{e}}) \subset (G_{l, K}^{\alg})^{\tau}$ then 
$\rho_{l} (\tilde\tau) \, G_{l, K_{e}}^{\alg} \subset (G_{l, K}^{\alg})^{\tau}.$ 
Because:
\begin{equation}
\rho_{l} (G_K)
 = \bigsqcup_{\tau \in \Gal(K_{e}/K)} \, \, 
 \rho_{l} (\tilde\tau) \, \rho_{l} (G_{K_{e}}),
\label{Im rhol GK decomposed cosets}
\end{equation}
then
\begin{equation}
G_{l, K}^{\alg}
 = \bigsqcup_{\tau \in \Gal(K_{e}/K)} \,\, \rho_{l} (\tilde\tau) \, G_{l, K_{e}}^{\alg}.
\label{decomposition of GlKalg into cosets of algebraic closures}
\end{equation} 

\noindent
This implies the following equalities:
\begin{equation}
G_{l, K}^{\alg}
 = \bigsqcup_{\tau \in \Gal(K_{e}/K)} \,\, (G_{l, K}^{\alg})^{\tau},
\label{decomposition of GlKalg into cosets of twisted algebraic closures}
\end{equation}  
\begin{equation}
G_{l, K, 1}^{\alg}
 = \bigsqcup_{\tau \in \Gal(K_{e}/K)} \,\, (G_{l, K, 1}^{\alg})^{\tau}.
\label{decomposition of Glk1 into twists}\end{equation} 

\noindent
Now we observe that
$\rho_{l} (\tilde\tau) \, G_{l, K_{e}}^{\alg} = (G_{l, K}^{\alg})^{\tau}$  for all $\tau.$ 
This implies the equality $(G_{l, K, 1}^{\alg})^{\id} = G_{l, K_e, 1}^{\alg}$ 
and the following natural isomorphism:

\begin{equation}
G_{l, K}^{\alg}/ (G_{l, K}^{\alg})^{\id} \cong \Gal(K_e/K).
\label{GlK1alg mod GlK1algid = GLeK}
\end{equation}

\noindent
Since $\gpDL_{K}^{\id}(V, \psi, D) = \gpDL_{K_{e}}(V, \psi, D) = \gpDL_{\overline{F}}(V, \psi, D)$, we get
\begin{equation}
G_{l, K_{e}, 1}^{\alg} \subseteq \gpDL_{K_{e}}(V, \psi, D)_{\Q_l}.
\label{GlLe1alg subset DLLeA}\end{equation} 

\noindent
Hence by \eqref{GlK1algtau subset DLKtauVpsiDQl},  \eqref{GlK1alg mod GlK1algid = GLeK} and Theorem \ref{equality of conn comp for Glalg and Glalg1}
there are natural isomorphisms:
\begin{equation}
G_{l, K, 1}^{\alg}/ (G_{l, K, 1}^{\alg})^{\id} \, \cong \,
\gpDL_{K}(V, \psi, D)/ \gpDL_{K}^{\id}(V, \psi, D) \, \cong \, \Gal(K_e/K).
\label{GK1alg mod GK1algID with resp. to DL mod DLId}\end{equation}
\end{remark}
\medskip

\begin{theorem}\label{connected components iso for DL}
The following equalities are equivalent: 
\begin{equation}
G_{l, K_{e}, 1}^{\alg} = \gpDL_{K_{e}}(V, \psi, D)_{\Q_l}.
\label{ItIsSatoTateGroupIdentification Ke}\end{equation}
\begin{equation}
G_{l, K, 1}^{\alg} = \gpDL_{K}(V, \psi, D)_{\Q_l}.
\label{ItIsSatoTateGroupIdentification}\end{equation}
Let $L/K$ be a finite extension such that $L \subset {\overline F}.$ 
The following equalities are equivalent: 
\begin{equation}
G_{l, L_e, 1}^{\alg} = \gpDL_{L_e}(V, \psi, D)_{\Q_l}.
\label{ItIsSatoTateGroupIdentification L}\end{equation}
\begin{equation}
G_{l, L, 1}^{\alg} = \gpDL_{L}(V, \psi, D)_{\Q_l}.
\label{ItIsSatoTateGroupIdentification Le}
\end{equation}
Moreover equalities \eqref{ItIsSatoTateGroupIdentification L} and
\eqref{ItIsSatoTateGroupIdentification Le} imply equalities 
\eqref{ItIsSatoTateGroupIdentification Ke} and
\eqref{ItIsSatoTateGroupIdentification}.
\end{theorem}

\begin{proof} The equivalence of \eqref{ItIsSatoTateGroupIdentification Ke}
and \eqref{ItIsSatoTateGroupIdentification} follows from 
\eqref{GK1alg mod GK1algID with resp. to DL mod DLId}.
Changing base to an extension $L/K$, the equivalence of 
\eqref{ItIsSatoTateGroupIdentification L} and 
\eqref{ItIsSatoTateGroupIdentification Le} also follows from
\eqref{GK1alg mod GK1algID with resp. to DL mod DLId}. 
Observe that $\Ker(\rho_e | G_L) \subset  \Ker  \rho_e.$
Hence $K_e \subset L_e.$ It follows that 
$\gpDL_{K_e}(V, \psi, D) = \gpDL_{L_e}(V, \psi, D)$ and 
$G_{l, L_e, 1}^{\alg} \subset G_{l, K_e, 1}^{\alg}.$
Hence \eqref{GlLe1alg subset DLLeA} and \eqref{ItIsSatoTateGroupIdentification L}
imply \eqref{ItIsSatoTateGroupIdentification Ke}.
\end{proof}

\section{Connected components of $\AST_K$ and $\ST_K$}

\begin{remark}\label{minimal field of connectedness}
Consider the continuous homomorphism
\begin{equation}
\epsilon_{l, K}\, :\, G_K \rightarrow G_{l, K}^{\alg}(\Q_l). 
\label{epsilon}\end{equation}
Since $\rho_{l} (G_K)$ is Zariski dense in $G_{l, K}^{\alg}$, this map induces the continuous epimorphism:
\begin{equation}
{\tilde{\epsilon}}_{l, K}\, :\, G_K \rightarrow \pi_{0}(G_{l, K}^{\alg}).
\label{tilde epsilon}\end{equation}

\noindent
Since $(G_{l, K}^{\alg})^{\circ}$ is open in $G_{l, K}^{\alg}$, we get:
\begin{equation}
\epsilon_{l, K}^{-1} ((G_{l, K}^{\alg})^{\circ}(\Q_l)) \,\, = \,\,
\Ker {\tilde\epsilon}_{l, K} \,\, = \,\, G_{K_0}
\label{epsilon and tilde epsilon}\end{equation} 
for some finite Galois extension $K_{0}/K.$ From Proposition
\ref{L0realizing conn comp} and Theorem 
\ref{equality of conn comp for Glalg and Glalg1}
it follows that $K_{0}/K$ is the minimal extension such that 
$G_{l, K_{0}}^{\alg} = (G_{l, K}^{\alg})^{\circ}$ and $G_{l, K_{0}, 1}^{\alg} = 
(G_{l, K, 1}^{\alg})^{\circ}$. In principle, $K_0$ may depend
on $l$; in Proposition \ref{minimal field of connectedness independent of l} below, 
we will give conditions for the independence of $K_0$ from $l.$
These conditions are satisfied in the case of abelian varieties; see Remark~\ref{Serre homotheties Theorem}.

Let $\tilde\sigma \in G_K$ be a lift of $\sigma \in \Gal(K_{0}/K).$ The coset
$\tilde\sigma \, G_{K_{0}}$ does not depend on the lift. By the definition of $K_0$, there is an obvious 
isomorphism:
\begin{equation}
G_{l, K}^{\alg}/ (G_{l, K}^{\alg})^{\circ} \cong \Gal(K_0/K).
\label{GlK1alg mod GlK1alg0 = GK0K}
\end{equation}

\noindent
Also, the Zariski closure of  
$\rho_{l} (\tilde\sigma \, G_{K_{0}}) = \rho_{l} (\tilde\sigma) \, \rho_{l} (G_{K_{0}})$ 
in $\GIso_{(V_l, \psi_l)}$ is  $\rho_{l} (\tilde\tau) \, G_{l, K_{0}}^{\alg}.$ Because
\begin{equation}
\rho_{l} (G_K)
 = \bigsqcup_{\sigma \in \Gal(K_{0}/K)} \, \, 
 \rho_{l} (\tilde\sigma) \, \rho_{l} (G_{K_{0}}),
\label{Im rhol GK decomposed cosets for K_0}
\end{equation}
by the definition of $K_0$ we have:
\begin{equation}
G_{l, K}^{\alg}
 = \bigsqcup_{\sigma \in \Gal(K_{0}/K)} \,\, \rho_{l} (\tilde\sigma) \, G_{l, K_{0}}^{\alg}.
\label{decomposition of GlKalg into cosets of algebraic closures with resp. to K0}
\end{equation}
\end{remark}

\begin{remark} \label{field of connectednes for GlK1alg}
Let $H_{l, K, 1} := \rho_{l}^{-1} (\rho_{l} (G_K)_1)$ and $K_1 := \overline{K}^{H_{l, K, 1}}.$
Observe that:
\begin{gather*}
\epsilon_{l, K}^{-1} ((G_{l, K, 1}^{\alg})^{\circ}(\Q_l)) = \epsilon_{l, K}^{-1} (G_{l, K_0, 1}^{\alg}(\Q_l)) =
\epsilon_{l, K}^{-1} ((G_{l, K_0}^{\alg} \cap \Iso_{(V_l, \psi_l)})(\Q_l)) = 
\\
= \epsilon_{l, K}^{-1} (G_{l, K_0}^{\alg}(\Q_l)) \,\cap\, 
\epsilon_{l, K}^{-1}(\Iso_{(V_l, \psi_l)}(\Q_l))
= G_{K_0} \cap G_{K_1} = G_{K_{0}K_1}.
\end{gather*}
\end{remark}

\begin{remark}
We observe that $K \subset K_e \subset K_0.$
\end{remark}

\begin{proposition} \label{connected component of ASTK} 
Assume Conjecture
\ref{general algebraic Sato Tate conj.} ${(\rm{a})}$ and 
assume that $\gpast_{l, K}$ and $\gpast_{l, K_{0}}$ are isomorphisms for a fixed $l.$ Let
$L/K_{0}$ be a finite Galois extension. Then: 
\begin{itemize}
\item[(1)] $\AST_{K_{0}} = (\AST_{K})^{\circ}.$
\item[(2)] $\ST_{K_{0}} = (\ST_{K})^{\circ}$ up to conjugation
in $\AST_{K} (\C).$ 
\item[(3)] $\AST_{K_{0}} = \AST_{L}.$ 
\item[(4)] $\ST_{K_{0}} = \ST_{L}$ up to conjugation
in $\AST_{K_{0}} (\C).$ 
\end{itemize}
\end{proposition}

\begin{proof}
Consider the following commutative diagram. The bottom row is exact. The right
vertical arrow is an isomorphism by \eqref{ConnCompIsom} of Proposition \ref{connected components iso}
(cf. Remark \ref{connected components iso for given l}).

\begin{equation}
{\xymatrix{
1 \ar[r]^{} & \,\, \AST_{K_{0},\, \Q_l}   \ar[r]^{}  
& \quad \AST_{K, \, \Q_l}  \ar[r]^{}  \quad & \,\,
\pi_{0} (\AST_{K, \, \Q_l}) \ar[r]^{} & 1\\ 
1 \ar[r]^{} & \,\, G_{l, K_{0}, 1}^{\alg} \ar@<0.1ex>[u]^{{\gpast_{l, K_{0}}}}_{\cong}  \ar[r]^{}  & 
\quad  G_{l, K, 1}^{\alg} \ar@<0.1ex>[u]^{{\gpast_{l, K}}}_{\cong}  \ar[r]^{}  \quad & 
\,\, \pi_{0} (G_{l, K, 1}^{\alg}) \ar@<0.1ex>[u]_{\cong} \ar[r]^{} & 1}
\label{diagram GlL1 maps to ASTK}}
\end{equation}

\noindent
Since $\AST_{K_{0}, \, \Q_l}$ is connected (since it is isomorphic to
$G_{l, K_{0}, 1}^{\alg}$), the exactness of the top row in 
(\ref{diagram GlL1 maps to ASTK}) implies 
$\AST_{K_{0}, \, \Q_l} = (\AST_{K, \, \Q_l})^{\circ}$ and in particular 
$\AST_{K_{0}} = (\AST_{K})^{\circ}.$ By Proposition 
\ref{connected components iso} we obtain $\pi_{0}(\ST_{K}) = \pi_{0} (\AST_{K})$ and $\pi_{0}(\ST_{K_{0}}) = \pi_{0} (\AST_{K_{0}}) = 1.$
Hence by (1) we have $\ST_{K_{0}} \subset (\ST_{K})^{\circ}$ up to conjugation
in $\AST_{K} (\C)$ because $\ST_{K_{0}}$ is connected and compact and $\ST_{K}$ maximal compact in 
$\AST_{K} (\C).$ On the other hand $(\ST_{K})^{\circ} \subset (\AST_{K})^{\circ} (\C) = \AST_{K_{0}}(\C)$ by 
$\pi_{0}(\ST_{K}) = \pi_{0} (\AST_{K})$ and by (1).
Hence $(\ST_{K})^{\circ} \subset  \ST_{K_{0}}$ up to conjugation in $\AST_{K_{0}} (\C)$ because 
$\ST_{K_{0}}$ is maximal compact in $\AST_{K_{0}} (\C).$ Hence (2) follows. To prove (3) observe that
$G_{l, K_{0}, 1}^{\alg} = G_{l, L, 1}^{\alg}$ because $G_{l, L, 1}^{\alg}$ is a normal subgroup of finite index in
$G_{l, K_{0}, 1}^{\alg}$ and $G_{l, K_{0}, 1}^{\alg}$ is connected. Then (3) follows from the following commutative diagram:
$${\xymatrix{
\quad G_{l, K_{0}, 1}^{\alg} \ar[r]^{\gpast_{l, K_{0}}}_{\cong} \quad 
& \quad  \AST_{K_{0},\,  \Q_l}  \\ 
\quad G_{l, L, 1}^{\alg} \ar[r]^{\gpast_{l, L}} \ar@<0.1ex>[u]^{=}  \quad 
& \quad   \AST_{L, \, \Q_l} \ar@<0.1ex>[u]^{}  \\
}}
\label{initial diagram for L over K0}$$
and (3) implies (4) directly. 
\end{proof}

\begin{proposition} \label{minimal field of connectedness independent of l}
Assume that Conjecture \ref{general algebraic Sato Tate conj.} holds for $K$ and $K_0.$ Then the field $K_0$ is independent of $l.$ 
\end{proposition}

\begin{proof}

\noindent 
Assume that the corresponding equality to \eqref{epsilon and tilde epsilon} holds for $l^{\prime}$ and
$K_{0}^{\prime}.$ Hence by Remark \ref{naturality of ast l K}, the assumptions and Proposition 
\ref{connected component of ASTK} we have 
$(\AST_{K})^{\circ} \cong \AST_{K_{0}} \cong \AST_{K_{0}^{\prime}}.$ 
Then from continuity of the maps ${\epsilon}_{l^{\prime}, K}$ and ${\tilde\epsilon}_{l^{\prime}, K}$ 
we find out that $K_{0}^{\prime} \subset K_0.$
By symmetry, from continuity of the maps ${\epsilon}_{l, K}$ and ${\tilde\epsilon}_{l, K}$ we obtain 
$K_0 \subset K_{0}^{\prime}.$ 
\end{proof}

\begin{remark}
Let $C \in \N$ be fixed. Then Proposition \ref{minimal field of connectedness independent of l} has the following version for all $l \geq C.$
\end{remark}

\begin{proposition} \label{minimal field of connectedness independent of l leq C}
Assume that for every $l \geq C$ the homomorphisms $\gpast_{l, K}$ and $\gpast_{l, K_{0}}$ are isomorphisms. Then the field $K_0$ is independent of $l \geq C.$ 
\end{proposition}
\medskip

The surjectivity of \eqref{diagram rho GLl1 maps to rho GKll and rho GLl maps to rho GK} is a 
subtle point in the computation of Sato-Tate groups. Below we find conditions for 
the surjectivity. Let $L/K$ be a finite Galois extension. Consider the following commutative diagram
where $\Zar_{L/K} := \Zar_{l, \, L/K}$ and $\Zar_{L/K, \, 1} := \Zar_{l, \, L/K, \, 1}.$

\begin{equation}
{\xymatrix{
\rho_{l} (G_K) / \rho_{l} (G_L)  \ar@{>>}[r]^{\Zar_{L/K}} \quad 
& \quad G_{l, K}^{\alg} / G_{l, L}^{\alg} \\ 
\rho_{l} (G_K)_1 / \rho_{l} (G_L)_1  \ar[r]^{\Zar_{L/K, \, 1}} \ar@<0.1ex>[u]^{j_{L/K}}    \quad 
& \quad  G_{l, K, 1}^{\alg} / G_{l, L, 1}^{\alg} \ar@<0.1ex>[u]_{\cong}^{i_{L/K}}  \\
}}
\label{diagram comparing the rho (GK) with GLalg}
\end{equation}
\medskip

\noindent
We put 
$$
\bar{l} \,\, = \,\,
\left\{
\begin{array}{lll}
l&\rm{if}&l > 2\\
8&\rm{if}&l=2.\\
\end{array}\right.
$$

\noindent
Let $K(\mu_{\bar{l}}^{\otimes n}) := {\bar K}^{{{\rm Ker}} \widetilde{\chi_{c}^{n}}},$
where $\widetilde{\chi_{c}^{n}} : G_K \rightarrow {\rm{Aut}} (\mu_{\bar{l}}^{\otimes n})$ is the $n$-th power of the cyclotomic character $\mod\bar{l}$.

\begin{lemma} Let $L/K$ be a finite Galois extension. Assume that:
\begin{itemize}
\item[(1)] $L \, \cap \, K(\mu_{\bar{l}}^{\otimes \, n}) \, = \, K$;
\item[(2)] $1 + l\Z_{l} \,\, {\rm{Id}}_{V_l} \, \subset \, \rho_{l} (G_K)$;
\item[(3)] $\Zar_{L/K}$ is an isomorphism.
\end{itemize}
Then the maps $j_{L/K}$ and $\Zar_{L/K, \, 1}$ are isomorphisms.
\label{JLK and Zar1 are isomorphisms}\end{lemma}
\begin{proof}
By assumption (3), the upper horizontal arrow in \eqref{diagram comparing the rho (GK) with GLalg} is an isomorphism. The left vertical arrow 
(see \eqref{diagram rho GLl1 maps to rho GKll and rho GLl maps to rho GK})
is a monomorphism, and by Theorem \ref{pre Serre theorem} the right vertical arrow is an isomorphism. 
To show the theorem, it is 
enough to prove that the bottom horizontal arrow $\Zar_{L/K, \, 1}$ is an epimorphism. 
For each $\sigma \in \Gal(L/K)$  we can choose, by assumption (1), 
a lift $\tilde{\sigma} \in \Gal({\overline F}/K)$ such that 
$$\tilde{\sigma}\, |_{K(\mu_{\bar{l}}^{\otimes \, n})} = 
{\rm{Id}}_{K(\mu_{\bar{l}}^{\otimes \, n})}.$$ 
Recall the natural exact sequence:
\begin{equation}
1 \rightarrow \Iso_{(V_l, \psi_l)} \rightarrow \GIso_{(V_l, \psi_l)} 
\,\, {\stackrel{\chi}{\longrightarrow}} \,\, \G_m \rightarrow 1.
\label{Iso_l defining exact sequence}
\end{equation}
Since $\rho_{l} (G_K) \subset G_{l, K}^{\alg} (\Q_l) \subset \GIso_{(V_l, \psi_l)} (\Q_l)$, the choice of the lift
$\tilde{\sigma}$ and the equality (\ref{comatibility of chi with chi-cycl.}) give
$\chi (\rho_{l} (\tilde{\sigma})) \in 1 + \bar{l} \, \Z_l \subset \G_{m} (\Q_l).$
Hence $\sqrt{\chi (\rho_{l} (\tilde{\sigma}))} \in 1 + l\Z_{l}$ because $(1 + l\Z_{l})^2 = 1 + \bar{l} \,\Z_{l}.$
By assumption (2),
there exists $\tilde{\gamma} \in G_K$ such that 
$\rho_l (\tilde{\gamma}) = \sqrt{\chi (\rho_{l} (\tilde{\sigma}))} \,\, {\rm{Id}}_{V_l}.$
By Remark \ref{the image of homothety via chi}, we have $\chi(\alpha \,\, {\text{Id}}_{V_l}) = \alpha^{2}$ for any $\alpha \in \Q_{l}^{\times}.$ Hence:
\begin{equation}
\chi(\rho_{l} (\tilde{\sigma} \tilde{\gamma}^{-1})) =  \chi(\rho_{l} (\tilde{\sigma}))
\chi(\rho_{l} (\tilde{\gamma}))^{-1}  = 1.
\end{equation}

\noindent
It follows that $\rho_{l} (\tilde{\sigma} \tilde{\gamma}^{-1}) \in \rho_{l} (G_K)_1.$      
Since: 
\begin{equation}
\rho_{l} (G_K)  \, = \bigcup_{\sigma \in \Gal(L/K)} \, \rho_{l} (\tilde{\sigma}) \,\rho_{l} (G_L) \, = \,
{{\bigsqcup}^{\prime}}_{\sigma \in \Gal(L/K)} \,\,\,\, \rho_{l} (\tilde{\sigma}) \,\rho_{l} (G_L)
\label{rholGK as a sum of cosets of rholGL}
\end{equation}
then by assumption (3):
\begin{equation}
G_{l, K}^{\alg}  \, = \, \bigcup_{\sigma \in \Gal(L/K)} \, \rho_{l} (\tilde{\sigma}) \, G_{l, L}^{\alg}
\,\, = \,\,
{{\bigsqcup}^{\prime}}_{\sigma \in \Gal(L/K)} \,\,\,\, \rho_{l} (\tilde{\sigma}) \, G_{l, L}^{\alg}.
\label{GKalg as a sum of cosets of GLalg}
\end{equation}
where ${{\bigsqcup}^{\prime}}_{\sigma \in \Gal(L/K)}$ is the summation over some set of $\sigma \in \Gal(L/K)$ such that
$\rho_{l} (\tilde{\sigma}) \,\rho_{l} (G_L)$ are all different cosets of $\rho_{l} (G_L)$ in 
$\rho_{l} (G_K).$
Because of (\ref{GKalg as a sum of cosets of GLalg}) we have $(G_{l, K}^{\alg})^{\circ} \, \subset \, G_{l, L}^{\alg}.$
It is obvious that $\G_m \, {\text{Id}}_{V_l} \subset (G_{l, K}^{\alg})^{\circ}.$ Hence 
$\rho_{l} (\tilde{\gamma})) \in (G_{l, K}^{\alg})^{\circ} \, \subset \, G_{l, L}^{\alg}.$ 
Hence $\rho_{l} (\tilde{\sigma}) \, G_{l, L}^{\alg} =  
i_{L/K} (\rho_{l} (\tilde{\sigma} \tilde{\gamma}^{-1}) \, G_{l, L, 1}^{\alg})$ and it follows that
$\Zar_{L/K, \,1}$ is an epimorphism.
\end{proof}

\begin{corollary} Let $L/K$ be a finite Galois extension. Assume that:
\begin{itemize}
\item[(1)] $L \, \cap \, K(\mu_{\bar{l}}^{\otimes \, n}) \, = \, K$;
\item[(2)] $1 + l\Z_{l} \,\, {\rm{Id}}_{V_l} \, \subset \, \rho_{l} (G_K)$;
\item[(3)] $\Zar_{L/K}$ is an isomorphism;
\item[(4)] $G_K / G_L \cong \rho_{l} (G_K) / \rho_{l} (G_L)$.
\end{itemize}
Then each coset of $G_{K}/G_{L}$ has the form $\tilde{\sigma}_{1} \, G_{L}$ such that:
\begin{itemize}
\item[(1)]  $\rho_{l} (\tilde{\sigma}_{1}) \in \rho_{l} (G_K)_1$;
\item[(2)]  $\rho_{l} (G_K)_1 \,\, = \,\, 
{\bigsqcup}_{\tilde{\sigma}_1  G_{L}}  \,\, \rho_{l} (\tilde{\sigma}_{1})\, \rho_{l} (G_L)_1$;
\item[(3)]  $G_{l, K, 1}^{\alg} \,\, = \,\, 
{\bigsqcup}_{\tilde{\sigma}_1  G_{L}}  \,\, \rho_{l} (\tilde{\sigma}_{1})\, G_{l, L, 1}^{\alg} $.
\end{itemize} 
\label{JLK and Zar1 are isomorphisms and tildesigma indep. of l}
\end{corollary}

\begin{proof}
Pick elements $\tilde{\sigma} \in G_K$ which represent all of the cosets of $G_L$ in $G_K$.
Because of assumption (4), we have: 
\begin{equation}
\rho_{l} (G_K) = \bigsqcup_{\tilde{\sigma} \, G_L} \,\, \rho_l (\tilde{\sigma}) \, \rho_{l} (G_L).
\label{rhoGK = bigsqcup rhoGL1}
\end{equation}
By Lemma \ref{JLK and Zar1 are isomorphisms}, the map $j_{L/K}$ is an isomorphism. Hence for every 
$\tilde{\sigma}$ there is $\tilde{\sigma}_1 \in G_K$ such that $\rho_{l} (\tilde{\sigma}_{1}) \in \rho_{l} (G_K)_1$
and
$\rho_{l} (\tilde{\sigma}) \rho_{l} (G_L) =
\rho_{l} (\tilde{\sigma}_1) \rho_{l} (G_L).$ By assumption (4) we obtain 
$\tilde{\sigma} \, G_L = \tilde{\sigma}_1 \, G_L.$ Since $j_{L/K}$ is an isomorphism, the claim (2)
holds. The claim (3) follows because $\Zar_{L/K, \, 1}$ is an isomorphism by Lemma \ref{JLK and Zar1 are isomorphisms}.
\end{proof}
\medskip

\begin{theorem} Assume that:
\begin{itemize}
\item[(1)] $K_e \, \cap \, K(\mu_{\bar{l}}^{\otimes \, n}) \, = \, K$;
\item[(2)] $1 + l\Z_{l} \,\, {\rm{Id}}_{V_l}  \, \subset \, \rho_{l} (G_K)$.
\end{itemize}
Then all arrows in the following commutative diagram are isomorphisms:
\begin{equation}
{\xymatrix{
\rho_{l} (G_K) / \rho_{l} (G_{K_e})  \ar[r]^{\Zar_{K_e / K}}_{\cong} \quad 
& \quad G_{l, K}^{\alg} / G_{l, K_e}^{\alg} \\ 
\rho_{l} (G_K)_1 / \rho_{l} (G_{K_e})_1  \ar[r]^{\Zar_{K_e / K, \, 1}}_{\cong} \ar@<0.1ex>[u]^{j_{K_e/K}}_{\cong}    \quad 
& \quad  G_{l, K, 1}^{\alg} / G_{l, K_{e}, 1}^{\alg} \ar@<0.1ex>[u]_{\cong}^{i_{K_e/K}} \\}}
\label{diagram comparing the rho (GK) with GLalg for Ke}
\end{equation}
\label{jKeK and Zar1 are isomorphisms}
\end{theorem}

\begin{proof}
By \eqref{Im rhol GK decomposed cosets} and 
\eqref{decomposition of GlKalg into cosets of algebraic closures}, the 
upper horizontal arrow $\Zar_{K_e / K}$ in diagram \eqref{diagram comparing the rho (GK) with GLalg for Ke}
is an isomorphism. Now the assumptions (1) and (2) and Lemma \ref{JLK and Zar1 are isomorphisms} show that 
all of the arrows in \eqref{diagram comparing the rho (GK) with GLalg for Ke} are isomorphisms.
\end{proof}

\begin{theorem} Assume that:
\begin{itemize}
\item[(1)] $K_0 \, \cap \, K(\mu_{\bar{l}}^{\otimes \, n}) \, = \, K,$
\item[(2)] $1 + l\Z_{l} \,\, {\rm{Id}}_{V_l}  \, \subset \, \rho_{l} (G_K).$
\end{itemize}
Then all arrows in the following commutative diagram are isomorphisms:
\begin{equation}
{\xymatrix{
\rho_{l} (G_K) / \rho_{l} (G_{K_0})  \ar[r]^{\Zar_{K_0/K}}_{\cong} \quad 
& \quad G_{l, K}^{\alg} / G_{l, K_0}^{\alg} \\ 
\rho_{l} (G_K)_1 / \rho_{l} (G_{K_0})_1  \ar[r]^{\Zar_{K_0/K,\, 1}}_{\cong} \ar@<0.1ex>[u]^{j_{K_0/K}}_{\cong}    \quad 
& \quad  G_{l, K, 1}^{\alg} / G_{l, K_{0}, 1}^{\alg} \ar@<0.1ex>[u]_{\cong}^{i_{K_0/K}} \\}}
\label{diagram comparing the rho (GK) with GLalg for K0}
\end{equation}
Moreover each coset of $G_{K}/G_{K_0}$ has the form $\tilde{\sigma}_{1} \, G_{K_0}$ such that:
\begin{itemize}
\item[(1)]  $\rho_{l} (\tilde{\sigma}_{1}) \in \rho_{l} (G_K)_1$;
\item[(2)]  $\rho_{l} (G_K)_1 \,\, = \,\, 
{\bigsqcup}_{\tilde{\sigma}_1  G_{K_0}}  \,\, \rho_{l} (\tilde{\sigma}_{1})\, \rho_{l} (G_{K_0})_1$;
\item[(3)]  $G_{l, K, 1}^{\alg} \,\, = \,\, 
{\bigsqcup}_{\tilde{\sigma}_1  G_{K_0}}  \,\, \rho_{l} (\tilde{\sigma}_{1})\, G_{l, K_0, 1}^{\alg} $.
\end{itemize} 
\label{jK0K and Zar1 are isomorphisms}
\end{theorem}

\begin{proof}
It follows from \eqref{Im rhol GK decomposed cosets for K_0} and 
\eqref{decomposition of GlKalg into cosets of algebraic closures with resp. to K0} that the 
upper horizontal arrow $\Zar_{K_0 / K}$ in diagram \eqref{diagram comparing the rho (GK) with GLalg for K0}
is an isomorphism.
Now the assumptions (1) and (2) and Lemma \ref{JLK and Zar1 are isomorphisms} show that 
all of the arrows in \eqref{diagram comparing the rho (GK) with GLalg for K0} are isomorphisms.
The isomorphism \eqref{Im rhol GK decomposed cosets for K_0} shows that the assumption (4) of Corollary 
\ref{JLK and Zar1 are isomorphisms and tildesigma indep. of l} is fulfilled, i.e., 
$G_K / G_{K_0} \cong \rho_{l} (G_K) / \rho_{l} (G_{K_0}).$ Hence the claims (1)--(3) 
follow by Corollary \ref{JLK and Zar1 are isomorphisms and tildesigma indep. of l}.  
\end{proof}

\begin{theorem} \label{STK iff STK0} Assume Conjecture \ref{general algebraic Sato Tate conj.} ${\rm{(a)}}$ and assume that for some $l:$
\begin{itemize}
\item[(1)] $K_0 \, \cap \, K(\mu_{\bar{l}}^{\otimes \, n}) \, = \, K$;
\item[(2)] $1 + l\Z_{l} \,\, {\rm{Id}}_{V_l}  \, \subset \, \rho_{l} (G_K)$;
\item[(3)] $\gpast_{l, K}$ is an isomorphism.
\end{itemize}
Then: 
\begin{equation}
\AST_{K, \, \Q_l}  \, = \, {\bigsqcup}_{\tilde{\sigma}_1 \, G_{K_0}}  \, 
\rho_{l} (\tilde{\sigma}_1) \, \AST_{K_{0}, \, \Q_l}
\label{decomposition of ASTKQl into cosets over Galois representatives}
\end{equation} 
\begin{equation}
\ST_{K}  \, = \, {\bigsqcup}_{\tilde{\sigma}_1 \, G_{K_0}}  \, 
\rho_{l} (\tilde{\sigma}_1) \, \ST_{K_{0}}
\label{decomposition of STK into cosets over Galois representatives}
\end{equation}
In particular the Sato-Tate conjecture (Conjecture \ref{general Sato Tate conj.}) on the equidistribution 
of normalized Frobenii in the representation $\rho_l$ with respect to $\ST_K$  holds if and only if 
the conjecture holds for the representation $\rho_{l} \, | \, G_{K_0}$ 
with respect to $\ST_{K_0}.$
\end{theorem}

\begin{proof} By Theorem \ref{jK0K and Zar1 are isomorphisms} we get
\begin{equation}
G_{l, K, 1}^{\alg}  \, = \, {\bigsqcup}_{\tilde{\sigma}_1 \,  G_{K_0}}  \, 
\rho_{l} (\tilde{\sigma}_1) \, G_{l, K_{0}, 1}^{\alg} \,\, .
\label{decomposition of GlK1alg into cosets over Galois representatives}
\end{equation}
Hence by Proposition \ref{connected components iso} we get the equality 
\eqref{decomposition of ASTKQl into cosets over Galois representatives}
which, under base change to $\C,$ taking $\C$-points and restricting to maximal compacts, 
gives the equality \eqref{decomposition of STK into cosets over Galois representatives}.
\end{proof}

Let us now specialize the previous discussion to abelian varieties.

\begin{remark}\label{Serre homotheties Theorem}
Fix an embedding of $K$ into $\C$.
Let $(V, \psi)$ be the Hodge structure associated to an abelian variety $A$ over $K$
(i.e., $n=1$, $V := H_1(A_{\C}, \Q)$, and $\psi$ is the pairing induced by a polarization of $A$).
Take $D$ to be $\End(A_{\overline{F}})_{\Q}$ (noting that this coincides with $D_h = D(V,\psi)$).
Let $T_l(A)$ be the $l$-adic Tate module of $A$ and let 
$V_l := V_l(A) := T_{l} (A) \otimes_{\Z_l} \Q_l.$ Let $\rho_l$ be the Galois representation 
of $G_K$ on $V_l$.
In this case, all the assumptions made in chapters 2--5 are satisfied,
and the resulting definitions agree with the corresponding definitions made in \cite{BK}.

 J.-P. Serre proved \cite{Se81} that 
the index $e(l)$ of the group of homotheties in $\rho_{l} (G_K)$ in the group of all homotheties 
is bounded when $l$ varies. Hence there is $c \in \N$ such that 
$(\Z_{l}^{\times})^c \,\, {\rm{Id}}_{V_l}  \, \subset \, \rho_{l} (G_K)$ for all $l.$ Hence 
for every $l$ coprime to $c$, we obtain $ 1 + l\, \Z_{l} \,\, {\rm{Id}}_{V_l}  \, \subset \, \rho_{l} (G_K).$ 
In this way, Serre established independence of $K_0$ from $l$; an explicit description of $K_0$ in terms of fields of definition of torsion points was later given by Larsen--Pink \cite{LP}.
\end{remark}

\begin{corollary}
With notation as in Remark~\ref{Serre homotheties Theorem},
suppose that $A/F$ satisfies the Mumford-Tate conjecture,
$\gpH(V,\psi) = \gpL(V,\psi,D)$, and $\gpDL_{K_e} (V,\psi,D)$ is connected.
Then for $l \gg 0$, the Sato-Tate conjecture holds for $A/K$  with respect to $\rho_l,$ 
if and only if the conjecture  holds for 
$A/K_{0}$ with repect to $\rho_{l} \, | \, G_{K_0}.$
\label{STK iff STK0 as far as MT for A holds}
\end{corollary}

\begin{proof} Obviously for $l \gg 0$ the condition
(1) of Theorem \ref{STK iff STK0} holds. The condition
(2) of Theorem \ref{STK iff STK0} holds for $l \gg 0$ by the result of
Serre \cite{Se81} discussed in Remark \ref{Serre homotheties Theorem} or
by the result of Wintenberger \cite[Corollary 1, p.\ 5]{W} showing the Lang conjecture.  
The condition (3) of Theorem \ref{STK iff STK0} holds by \cite[Theorem 6.1]{BK}. 
\end{proof}

\begin{corollary} \label{base change on Sato-Tate based on dimension}
With notation as in Remark~\ref{Serre homotheties Theorem},
put $g := \dim \, A,$ and let $E$ be the center of $D$. 
Assume that either $g \leq 3$ or $A$ is absolutely simple of type I, II or III in 
the Albert classification with $\frac{g}{d e}$ odd, where $d^2 = [D:E]$ and $e := [E:\, \Q].$ 
Then for $l \gg 0$, the Sato-Tate conjecture holds for $A/K$  with respect to $\rho_l,$ 
if and only if the conjecture  holds for 
$A/K_{0}$ with repect to $\rho_{l} \, | \, G_{K_0}$.
\label{STK iff STK0 as far as MT for A holds, case I, II, III and dim A less than 3}  
\end{corollary}
\begin{proof}  By \cite[Theorem 7.12, Cor. 7.19]{BGK1}, \cite[Theorem 5.11, Cor. 5.19]{BGK2} and \cite[Theorem 6.11]{BK},
abelian varieties considered in this corollary satisfy the Mumford-Tate conjecture and the 
properties: $\gpH(A) = \gpL(A)$ and $\gpDL_{K_e} (A)$ connected. Hence the corollary follows by 
Corollary \ref{STK iff STK0 as far as MT for A holds}.
\end{proof}

\begin{remark}
Some additional cases for which the conclusion of Corollary~\ref{base change on Sato-Tate based on dimension} holds are provided by the Jacobians of (certain) hyperelliptic curves,
thanks to the work of Zarhin \cite{Z1, Z2}.
\end{remark}

\section{Mumford-Tate group and Mumford-Tate conjecture}

For $A$ an abelian variety over $K$ and $(V_A, \psi_A)$ the associated polarized Hodge structure (as in Remark~\ref{Serre homotheties Theorem}), 
there is the following result.

\begin{theorem}
\label{Theorem Deligne} 
(Deligne \cite[I, Prop.\ 6.2]{D1}, Piatetski-Shapiro \cite{P-S}, Borovoi \cite{Bor}; see also \cite[\S 4.1]{Se77})
For any prime number $l$,
\begin{equation}
(G_{l, K}^{\alg})^{\circ} \subseteq \MT(V_A, \psi_A)_{\Q_l}.
\label{Del ineq}\end{equation}
\end{theorem}

\noindent
The classical conjecture for $A/K$ states:
\begin{conjecture} (Mumford-Tate) \label{Mumford-Tate for A} 
For any prime number $l$,
\begin{equation}
(G_{l, K}^{\alg})^{\circ} = \MT(V_A, \psi_A)_{\Q_l}.
\label{MT eq for A}\end{equation}
\end{conjecture}

\noindent
There is a general Mumford-Tate conjecture in the context of Hodge structures associated with 
$l$-adic representations \cite{UY}.

\begin{conjecture} (Mumford-Tate) \label{Mumford-Tate} 
For any prime number $l$,
\begin{equation}
(G_{l, K}^{\alg})^{\circ} = \MT(V, \psi)_{\Q_l}.
\label{MT eq}\end{equation}
\end{conjecture}

\begin{remark}\label{Mumford-Tate Conj and Deligne Theorem analogue} 
Assume that analogously to (\ref{Del ineq}) there is the following inclusion:
\begin{equation}
(G_{l, K}^{\alg})^{\circ} \subseteq \MT(V, \psi)_{\Q_l}.
\label{Del ineq general}\end{equation}
We see that (\ref{Del ineq general}) is equivalent to the inclusion
\begin{equation}
(G_{l, K, 1}^{\alg})^{\circ} \subseteq \gpH (V, \psi)_{\Q_l},
\label{Del ineq1}\end{equation}
while the Mumford-Tate conjecture is equivalent to the equality
\begin{equation}
(G_{l, K, 1}^{\alg})^{\circ} = \gpH(V, \psi)_{\Q_l}.
\label{H eq}\end{equation}
This follows immediately from the following commutative diagram
in which every column is exact and every horizontal arrow is a containment 
of corresponding group schemes. (Recall that $n$ is odd.)
$$
\xymatrix{
\quad 1 \ar@<0.1ex>[d]^{}  \quad 
& \quad 1 \ar@<0.1ex>[d]^{}  \quad 
&  1\ar@<0.1ex>[d]^{}\\
(G_{l, K, 1}^{\alg})^{\circ} \ar@<0.1ex>[d]^{} \ar[r]^{}  \quad 
& \quad \gpDH(V, \psi)_{\Q_l}   
\ar@<0.1ex>[d]^{} \ar[r]^{}  \quad & 
\quad \Iso_{V_{l}, \psi_{l}} \ar@<0.1ex>[d]^{} \\ 
(G_{l, K}^{\alg})^{\circ} \ar@<0.1ex>[d]^{} \ar[r]^{} \quad & 
\quad \MT (V, \psi)_{\Q_l}  \ar@<0.1ex>[d]^{} \ar[r]^{}  \quad & 
\quad  \GIso_{V_{l}, \psi_{l}} \ar@<0.1ex>[d]^{} \\
\quad \G_{m} \ar[r]^{=} \ar@<0.1ex>[d]^{}  \quad & 
\quad  \G_{m} \ar[r]^{=} \ar@<0.1ex>[d]^{}  \quad & 
\quad  \G_{m} \ar@<0.1ex>[d]^{} \\
\quad 1   \quad & \quad 1   \quad &  1}
\label{diagram in Deligne theorem}
$$
\end{remark}

The Mumford-Tate and Hodge groups do not behave well
in general with respect to products of Hodge structures, as can be seen in the case of abelian 
varieties \cite[p.\ 316]{G}. However, one has the following simple and well-known result
(see for instance \cite[(4.10)]{Mo2}); for more detailed discussions of products,
see any of \cite{Ha, MZ, Z3}. 

\begin{theorem}\label{Mumford-Tate group standard properties} 
The Mumford-Tate groups of Hodge structures have the following properties.
\begin{itemize} 
\item[1.] An isomorphism of rational, polarized Hodge structures  $\alpha : (V_1, \psi_1) 
\rightarrow (V_2, \psi_2) $  induces isomorphisms
$\MT (V_1, \psi_1) \cong \MT (V_2, \psi_2)$  and $\gpH (V_1, \psi_1) \cong \gpH (V_2, \psi_2).$ 
\item[2.] For $(V, \psi)$ is a rational, polarized Hodge structure, let
$(V, \psi)^s := \prod_{i=1}^s (V, \psi)$.
Then $\MT ((V, \psi)^s) \cong \MT ((V, \psi))$ and $\gpH ((V, \psi)^s) \cong \gpH ((V, \psi))$.
\end{itemize}
\end{theorem}

One can make a corresponding calculation also on the Galois side.
\begin{theorem}\label{Zariski closure standard properties} 
We have the following results.
\begin{itemize} 
\item[1.] An isomorphism $\phi : (V_{1, l}, \psi_{1, l}) \rightarrow 
(V_{2, l}, \psi_{2, l})$ of $\Q_l [G_F]$-modules induces isomorphisms: 
$G_{l, K}^{\alg} (V_{1,l}, \psi_{1, l}) \cong 
G_{l, K}^{\alg} (V_{2, l}, \psi_{2, l})$ and 
$G_{l, K, 1}^{\alg} (V_{1, l}, \psi_{1, l}) \cong 
G_{l, K, 1}^{\alg} (V_{2, l}, \psi_{2, l}).$ 
\item[2.] If $(V_l, \psi_l)$ is a $\Q_l [G_F]$-module then
for any positive integer $s$,
$G_{l, K}^{\alg} (V_{l}^s, \psi_{l}^s) \cong G_{l, K}^{\alg}  (V_l, \psi_{l})$ and 
$G_{l, K, 1}^{\alg}  (V_{l}^s, \psi_{l}^s) = G_{l, K, 1}^{\alg}  (V_l, \psi_{l}).$
\end{itemize}
\end{theorem}

\begin{proof}
1. Obvious.

\noindent
2. There is a natural isomorphism $\rho_{l, V_{l}^s} \cong \Delta \rho_{l, V_l}$ in which $\Delta \rho_{l, V_l} \, :\, G_K \rightarrow 
\GIso( (V_l)^s, \psi_{l}^s)$ is the natural diagonal representation
$\Delta \rho_{l, V_l} = {{\rm diag}} (\rho_{l, V_l}, \dots, \rho_{l, V_l}).$ 
Hence 
\begin{equation}
\rho_{l, V_{l}^s} (G_K) \cong {\Delta} \rho_{l, V_l} (G_K) \cong \rho_{l, V_l} (G_K),
\nonumber\end{equation}
This gives 
$$G_{l, K}^{\alg} (V_{l}^s, \psi_{l}^s)
 \cong  \Delta \, G_{l, K}^{\alg}  (V_l, \psi_{l}) \cong G_{l, K}^{\alg}  (V_l, \psi_{l}).$$
Moreover
$$G_{l, K, 1}^{\alg}  (V_{l}^s, \psi_{l}^s) = G_{l, K}^{\alg} (V_{l}^s, \psi_{l}^s) \cap
\Iso_{( (V_l)^s, \psi_{l}^s)} \cong 
\Delta \, G_{l, K}^{\alg} (V_{l}, \psi_{l}) \cap
\Iso_{( (V_l)^s, \psi_{l}^s)}$$
$$ \cong G_{l, K}^{\alg} (V_{l}, \psi_{l}) \cap
\Iso_{( (V_l), \psi_{l})} = G_{l, K, 1}^{\alg}  (V_l, \psi_{l}).$$
\end{proof}

\begin{corollary}\label{MT true for V implies true for Vs}
If the Mumford-Tate conjecture holds for $V$ then it
holds for $V^s$ for any positive integer $s.$
\end{corollary}
\begin{proof} It follows from Theorems \ref{Mumford-Tate group standard properties}  
and \ref{Zariski closure standard properties}. 
\end{proof}

\begin{remark}\label{Mumford-Tate Conj AST conjecture} 
Observe that if the Mumford-Tate conjecture holds for $(V, \psi)$
and $K$ is such that $G_{l, K}^{\alg}$ is connected, then for any $s \geq 1:$
\begin{equation}
G_{l, K, 1}^{\alg}((V, \psi)^s) = \gpH((V, \psi)^s)_{\Q_l}.
\label{H eq1}\end{equation}

\noindent
Hence the algebraic Sato-Tate conjecture holds for $(V, \psi)^s$ for any $s \geq 1$ with
\begin{equation}
\AST_{K} ((V, \psi)^s) = \gpH ((V, \psi)^s).
\label{AST conj holds when MT conj and GL ALG connect }
\end{equation}
\end{remark}


\section{Some conditions for the algebraic Sato-Tate conjecture}

Let $A$ be an abelian variety over $K$ and let
$D_A := {\rm End}_{\overline F}\,( A) \otimes_{\Z} \, \Q.$
For the polarized Hodge structure $(V_{A}, \psi_{A})$, the inclusion
$\gpH(V_{A}, \psi_A) \subseteq \gpL(V_A, \psi_A, D_A)$ can be strict, 
which makes the Mumford-Tate conjecture a subtle problem. Mumford
\cite{M} exhibited examples of simple 
abelian fourfolds for which 
$\gpH(V_{A}, \psi_A) \neq \gpL(V_A, \psi_A, D_A)$.
These examples have trivial endomorphism ring, but
the construction was generalized by Pohlmann \cite{Poh} to include some abelian varieties of CM type
(see \cite{MZ} for further discussion). Notwithstanding such constructions, in many cases where $A$ has a large endomorphism algebra as compared to its dimension
(e.g., under the hypotheses of Corollary~\ref{base change on Sato-Tate based on dimension}),
one can show that $\gpH(V_{A}, \psi_A) = \gpL(V_A, \psi_A, D_A)$ and that the Mumford-Tate conjecture holds. 

Returning to the general case,
let  $D \subset \End_{\Q} (V)$ be a $\Q$-subalgebra. Let $D$ admit a continuous $G_F$-action.
Let $(V, \psi)$ be a $D$-equivariant, polarized Hodge structure. Let $(V_l, \psi_l) := 
(V \otimes_{\Q} \Q_l, \psi \otimes_{\Q} \Q_l)$ be a family of Galois representations associated with 
the polarized Hodge structure $(V, \psi)$. In this chapter, we assume that the inclusion
(\ref{Del ineq general}) holds. In this setting, we say that the 
Mumford-Tate conjecture for $(V, \psi)$ is \emph{explained by endomorphisms} if 
the Mumford-Tate conjecture holds and $\gpH(V, \psi) = \gpL(V, \psi, D).$ 
The following theorem asserts that in cases where the Mumford-Tate conjecture is explained by endomorphisms
\emph{and} the twisted decomposable Lefschetz group over $\overline{F}$ is connected, the algebraic Sato-Tate
conjecture is in a sense also explained by endomorphisms.

\begin{theorem}\label{conditions for AST} 
Assume that the following conditions hold.
\begin{itemize} 
\item[1.] $\gpH(V, \psi) = \gpL(V, \psi, D) = \gpDL_{K_e}(V, \psi, D)$.
\item[2.] $(G_{l, K}^{\alg})^{\circ} = \MT(V, \psi)_{\Q_l}.$
\end{itemize}
Then (\ref{ItIsSatoTateGroupIdentification}) holds for every $l.$ 
Consequently, the algebraic Sato-Tate conjecture 
(Conjecture~\ref{general algebraic Sato Tate conj.})
holds for $(V, \psi)$ with
\begin{equation} 
\AST_K (V, \psi) = \gpDL_{K}(V, \psi, D).
\label{SatoTateGroupIdentified}
\end{equation}
\end{theorem}

\begin{proof} 
It is enough to prove (\ref{ItIsSatoTateGroupIdentification Le}).
By our assumptions and Remark \ref{Mumford-Tate Conj and Deligne Theorem analogue},
we get
$(G_{l, K_{e}, 1}^{\alg})^{\circ} = \gpH (V, \psi)_{\Q_l} = \gpL(V, \psi, D)_{\Q_l} = 
\gpDL_{K_{e}}(V, \psi, D)_{\Q_l}.$ 
It follows that  $\gpDL_{K_{e}}(V, \psi, D)_{\Q_l}$ is also connected 
for every $l$, and by (\ref{GlLe1alg subset DLLeA}) we obtain 
$(G_{l, K_{e}, 1}^{\alg})^{\circ} = G_{l, K_{e}, 1}^{\alg}$ for every $l.$ 
\end{proof}

\begin{remark} Under the assumptions of Theorem \ref{conditions for AST}, the results of 
Theorems \ref{Mumford-Tate group standard properties},  \ref{Zariski closure standard properties} and \ref{Decomposible Lefschetz group standard properties} show that the algebraic Sato-Tate conjecture holds for $V^s$ for all $s \geq 1$
with 
$\AST_K (V^s, \psi^s) = \gpDL_{K}(V^s, \psi^s, M_{s} (D)) \cong \gpDL_{K}(V, \psi, D) = 
\AST_K (V, \psi, D).$
\end{remark}

Conversely, if the algebraic Sato-Tate conjecture for $(V, \psi, D)$ is explained by endomorphisms,
so is the Mumford-Tate conjecture.
\begin{theorem}\label{conditions for MT via AST}
Assume that
(\ref{ItIsSatoTateGroupIdentification}) and (\ref{SatoTateGroupIdentified}) 
hold for every $l$ (so in particular, the algebraic Sato-Tate conjecture holds).
Moreover, assume that 
$(G_{l, K}^{\alg})^{\circ} \subset \MT(V, \psi)_{\Q_l}$. We then have the following.
\begin{itemize} 
\item[1.] $\gpH(V, \psi) = \gpL(V, \psi, D)$.
\item[2.] $(G_{l, K}^{\alg})^{\circ} = \MT(V, \psi)_{\Q_l}$. 
\end{itemize}
\end{theorem}
\begin{proof} 
By our assumptions and Remark  \ref{Mumford-Tate Conj and Deligne Theorem analogue} 
(see \eqref{Del ineq1}), we have
\begin{equation}
(G_{l, K, 1}^{\alg})^{\circ} \subseteq \gpH(V, \psi)_{\Q_l} \subseteq \gpL(V, \psi, D)_{\Q_l} 
= \gpDL_K(V, \psi, D)_{\Q_l}^{\circ}.
\label{GlK1 connected  and HA and LKA}
\end{equation}  
By (\ref{SatoTateGroupIdentified}) we get
\begin{equation}
(G_{l, K, 1}^{\alg})^{\circ} = \gpH(V, \psi)_{\Q_l} = \gpL(V, \psi, D)_{\Q_l} 
= \gpDL_K(V, \psi, D)_{\Q_l}^{\circ}.
\label{GlK1 connected and SatoTate gives HAl equal LKAl}
\end{equation}    
Hence by Remark \ref{Mumford-Tate Conj and Deligne Theorem analogue}, we obtain
$(G_{l, K}^{\alg})^{\circ} = \MT(V, \psi)_{\Q_l}.$
Moreover, since $\gpH(V, \psi)$ is closed in $\gpL (V, \psi, D)$, 
(\ref{GlK1 connected and SatoTate gives HAl equal LKAl}) gives
\begin{equation}
\gpH(V, \psi) = \gpL(V, \psi, D).
\label{GlK1 connected and SatoTate gives HA equal LKA}
\end{equation}
\end{proof}

\begin{remark}
Recall that 
\begin{equation}
\gpL (V, \psi, D) = \gpDL_{K} (V, \psi, D)^{\circ} 
\,\, \triangleleft \,\, \gpDL_{K}^{\id} (V, \psi, D) 
\,\, \triangleleft\,\,  \gpDL_{K} (V, \psi, D).
\end{equation}
Consider the following epimorphism of groups:
\begin{equation}
\gpDL_{K} (V, \psi, D) / \gpL(V, \psi, D) \, \rightarrow \, \gpDL_{K} (V, \psi, D) / \gpDL_{K}^{\id}(V, \psi, D) 
\, \cong \, G (K_e / K).
\label{DLKA to Galois Le over K}\end{equation}
If $(V, \psi, D)$ satisfies the assumptions of Theorem 
\ref{conditions for AST}, then the epimorphism 
(\ref{DLKA to Galois Le over K}) is an isomorphism. 
In this case we have an identification 
\begin{equation}
\pi_{0} (\AST_K(V, \psi, D) ) \, \cong \, \Gal(K_{e}/K).
\label{connected components of AST as Galois L0 over K}
\end{equation}
\end{remark}


\section{Motivic Galois group and motivic Serre group}

In the following sections we will give construction of the general algebraic Sato-Tate group 
in the category of motives for absolute Hodge cycles. See \cite{DM} (cf. \cite{Ja90}, \cite{Pan}, 
\cite{Se94}) concerning the construction and properties of the 
category of motives for absolute Hodge cycles. We will also make the $\ell$-adic realization of 
this construction explicit, and show that if a suitably motivic form of the Mumford-Tate conjecture 
holds then the algebraic Sato-Tate conjecture holds as well.

\begin{remark}
The category of motives for 
absolute Hodge cycles enjoys very nice properties: it is a semisimple abelian category and 
its Hom's are finite-dimen\-sional $\Q$-vector spaces. It is mainly due to the fact that the definition 
of Hom's is explained via the Betti, \' etale and de Rham realizations \cite[Prop. 6.1, p.\ 197]{DM}.
The advantage of use of this category of motives is that we do not need to assume standard conjectures
in our constructions.  
\end{remark}

\begin{definition}
Let $K$ be a number field. Choose an embedding of $K$ into $\overline{K}$. 
Let $\mathcal{M}_{K}$ (resp. $\mathcal{M}_{\overline{K}}$) (see \cite{DM})
be the motivic category for absolute Hodge cycles over $K$   (resp. ${\overline{K}}$). 
The Betti realization defines the fiber functor $H_{B}$ :
\begin{equation}
H_{B} : \mathcal{M}_{K} \rightarrow {\rm Vec}_{\Q}.
\label{fiber functor1}\end{equation}
\end{definition}
\medskip

The functor $H_B$ factors through the functor
\begin{equation}
\mathcal{M}_{K} \rightarrow 
\mathcal{M}_{\overline{K}}, \qquad
M \mapsto \overline{M} := M \otimes_{K} \overline{K}.
\label{tensor over K with overline K}
\end{equation}

\noindent
For $M \in 
\mathcal{M}_{K}$ let $\mathcal{M}_{K} (M)$ denote the smallest Tannakian subcategory of 
$\mathcal{M}_{K}$ containing $M.$ 
Let $H_{B} | \mathcal{M}_K (M)$ be the restriction of $H_B$ to 
$\mathcal{M}_K (M).$

\begin{definition}
The motivic Galois groups are defined as follows \cite{DM}, \cite{Se94}:
\begin{equation}
G_{\mathcal{M}_{K}} := {\rm Aut}^{\otimes} (H_{B}),
\label{Motivic Galois group I}\end{equation}
\begin{equation}
G_{\mathcal{M}_{K} (M)} := {\rm Aut}^{\otimes} (H_{B} | \mathcal{M}_{K} (M)). 
\label{Motivic Galois group II}\end{equation}
\end{definition}
\medskip

The algebraic groups $G_{\mathcal{M}_{K} (M)}$ are reductive
but not necessarily connected (see \cite[Prop.\ 2.23, p.\ 141]{DM}, 
cf. \cite[Prop.\ 6.23, p.\ 214]{DM}, \cite[p.\ 379]{Se94}). 
Observe that the finite-dimensional $\Q$-vector space 
${{\rm Hom}}_{\mathcal{M}_{\overline{K}}} 
(\overline{M}, \, \overline{N}) \in \mathcal{M}_{K}^{0}$ 
is a discrete $G_K$-module, so we consider it as an Artin motive.
Recall that $\mathcal{M}_{K}^{0}$ is equivalent to ${\rm Rep}_{\Q} (G_K)$,
the category of finite-dimensional $\Q$-vector spaces with continuous actions of $G_K.$

\begin{definition}
Fix a motive $M$ and put:
\begin{equation}
D := D (M) := {{\rm End}}_{\mathcal{M}_{\overline{K}}} 
(\overline{M})
\label{The ring of endomorphisms of a motive}\end{equation}
Let $h^{0}(D)$ denote the Artin motive corresponding to $D.$ 
Let $\mathcal{M}_{K}^{0} (D)$ be the smallest Tannakian subcategory of 
$\mathcal{M}_{K}^{0}$ containing $h^{0}(D)$ and put:
\begin{equation} 
G_{\mathcal{M}_{K}^{0} (D)} := {\rm Aut}^{\otimes} (H_{B}^{0} | \mathcal{M}_{K}^{0} (D)).
\label{Motivic Galois group II for D}
\end{equation}
\label{def of h0D}\end{definition}

There is a natural embedding of motives \cite[p.\ 215]{DM}, \cite[p.\ 53]{Ja90}: 
\begin{equation}
h^0(D) \subset \underline{{\rm End}}_{\mathcal{M}_{K} (M)} (M) = 
\underline{{\rm End}}_{\mathcal{M}_{K}} (M).
\label{h0D imbeds into end h1 A}\end{equation}
Recall that ${\underline{\rm End}}_{\mathcal{M}_{K} (M)} (M) = 
M^{\vee} \otimes M\in \mathcal{M}_{K} (M).$ 
In addition $G_{\mathcal{M}_{K}^{0}} \cong G_K$, so we observe that
\begin{equation}
G_{\mathcal{M}_{K}^{0} (D)} \cong \Gal(K_{e}/K).
\nonumber\end{equation}
Since $\mathcal{M}_{K}$ is semisimple \cite[Prop.\ 6.5]{DM} and $\mathcal{M}_{K} (M)$ is a strictly full
subcategory of $\mathcal{M}_{K}$, the motive $h^0(D) $ splits 
off of 
$\underline{{\rm End}}_{\mathcal{M}_{K} (M)} (M)$ in $\mathcal{M}_{K}.$
Moreover the semisimplicity of $\mathcal{M}_{K}$, together with the observation that 
$\mathcal{M}_{K}^{0}$ and $\mathcal{M}_K (M)$ are strictly full subcategories of
$\mathcal{M}_K$, shows that the top horizontal and left vertical maps in the following diagram 
are faithfully flat
(see \cite[(2.29)]{DM}):
\begin{equation}
{\xymatrix{
G_{\mathcal{M}_{K}} \ar@{>>}[d]^{} \ar@{>>}[r]^{}  & G_{K} \ar@{>>}[d]^{} \\
G_{\mathcal{M}_{K} (M)} \ar@{>>}[r]^{}  & \Gal(K_e/K)}}
\label{surjection of motivic Galois groups}
\end{equation}
In particular all homomorphisms in \eqref{surjection of motivic Galois groups} are surjective.
\bigskip

In the construction of $\mathcal{M}_{K}$ \cite[p.200--203]{DM} one starts with {{\bf effective motives}}
$h(X)$ and morphisms between them:
\begin{equation}
{\rm{Hom}}_{\mathcal{M}_{K}} (h(X), \, h(Y)) := 
{\rm{Mor}}^{0}_{\AH} (X, \, Y) := \CH_{\AH}^d (X \times Y)
\label{definition of Hom in MK}\end{equation}
where $X$ and $Y$ are smooth projective over $K$ and $X$ is of pure dimension $d.$ 
This leads swiftly (via Karoubian envelope construction etc.) to
the definition of the motivic category for absolute Hodge cycles $\mathcal{M}_{K}.$ In particular 
${\rm{Hom}}_{\mathcal{M}_{K}} (M, \, N),$ for any 
$M, N \in \mathcal{M}_{K},$ are relatively easy to handle. The obvious grading of the cohomology ring
brings the decomposition of the identity on $h (X)$ into a sum of the 
natural projectors:
\begin{equation}
{\rm{id}}_{h(X)} \, = \, \sum_{i \geq 0} \,\, \pi^i 
\label{decomposition of id on h(X) into projectors p i}\end{equation} 
As a result we get the natural decomposition 
\cite[p. 201--202]{DM}:
\begin{equation}
h(X) \, = \, \bigoplus_{i \geq 0} \,\, h^i (X) 
\label{decomposition of hX in to dir sum of hiM}\end{equation}
where $ h^i (X) := ( h (X), \, \pi^i).$ See also \cite{Ja90} and \cite{Pan}
for additional information about $\mathcal{M}_{K}.$
\medskip




Since $\mathcal{M}_{K}$ is abelian and semisimple, every motive $M \in \mathcal{M}_{K}$ 
is a direct summand of $h (X) (m),$ the twist of $h(X)$ by the $m$-th power of the 
Lefschetz motive $\LL := h^2 (\PP^1)$ for some $m \in \Z.$ The direct summands of 
motives of the form $h^r (X) (m)$ will be called \emph{homogeneous motives}.
Let $L/K$ be a field extension such that 
$K \subset L \subset \overline{K}.$ 
Then in $\mathcal{M}_{L}$, the motive $M \otimes_{K} L \in \mathcal{M}_{L}$ 
is a direct summand of the motive $h (X \otimes_{K} L) (m).$ 

 
\medskip


Observe that $H_{B} | \mathcal{M}_{K} (h^r (X)) \,\, (h^r (X)) \, = \,  H_{B} (h^r (X))  
\, = \, H^{r} (X(\C), \, \Q)$
and $V := H^{r} (X(\C), \, \Q)$ admits a $\Q$-rational polarized 
Hodge structure of weight $r$ with polarization $\psi^r$. 
The polarization comes up as follows. It is shown in \cite[pp.\ 197--199]{DM} (cf. \cite[p.\ 478--480]{Pan}, \cite[pp.\ 2--4]{Ja90})   
that if $\dim \, X = d$, then there is an element
$\psi^r \in \CH^{2d-r}_{\AH} (X \times X)$ such that for every embedding 
$\sigma  : K \hookrightarrow \C$, $\psi^r$ induces a $\Q$-bilinear map:
\begin{equation}
\psi^r \,\, : \,\, H^{r}_{\sigma} (X (\C), \, \Q) \times 
H^{r}_{\sigma} (X (\C), \, \Q) \rightarrow \Q (-r)
\label{Q Hodge structure for motives} 
\end{equation} 
which gives the polarization $\psi_{\R}^{r} := \psi^{r} \otimes_{\Q} \, \R$ of the real Hodge structure:
\begin{equation}
\psi_{\R}^{r} \,\, : \,\, H^{r}_{\sigma} (X (\C), \, \R) \times H^{r}_{\sigma} (X (\C), \, \R) \rightarrow 
\R (-r)
\label{R Hodge structure for motives}.
\end{equation} 
It is then shown
\cite[Prop.\ 6.1 (e), p.\ 197]{DM} that the Hodge decomposition of 
$V \otimes_{\Q} \C$ is $D = D(M)$-equivariant for $M = h^r(X).$ 
\medskip

In effect, for any homogeneous motive $M \in \mathcal{M}_{K},$ 
this induces the polarization of the real Hodge structure associated with the 
rational Hodge structure on the Betti realization $V := H_B (M).$  
The Hodge decomposition of $V \otimes_{\Q} \C$ is again $D = D(M)$-equivariant.
\bigskip

From now on in this paper, $M$ will always denote a homogeneous motive. 

\medskip

By the definition and properties of ${\rm Aut}^{\otimes} (H_{B} | \mathcal{M}_{K} (M))$,
cf.\ \cite[p.\ 128--130]{DM} and computations in \cite[p.\ 198--199]{DM}, we have:

\begin{equation}
G_{\mathcal{M}_{K}(M)} \subset \GIso_{(V, \psi)}. 
\label{GMKA subset of GSpV}\end{equation}

\begin{definition}
Define the following algebraic groups:
\begin{align*}
G_{\mathcal{M}_{K}(M), 1}  &:= G_{\mathcal{M}_{K}(M)} \cap \Iso_{(V, \psi)} \\
G_{\mathcal{M}_{K}(M), 1}^{\circ} &:= (G_{\mathcal{M}_{K}(M)})^{\circ} \cap 
\Iso_{(V, \psi)}.
\end{align*}
The algebraic group $G_{\mathcal{M}_{K}(M), 1}$ will be called 
the \emph{motivic Serre group}.
\label{motivic Galois group, motivic Serre group}\end{definition}
\medskip

\begin{remark}
Serre denotes the group $G_{\mathcal{M}_{K}(M), 1}$ by $G_{\mathcal{M}_{K}(M)}^1$
\cite[p.\ 396]{Se94}. 
\end{remark}

\begin{definition}
For any $\tau \in \Gal(K_{e}/K)$, put
\begin{equation}
\GIso_{(V, \psi)}^{\tau} := 
\{g \in \GIso_{(V, \psi)}:\,\, g \beta g^{-1} = 
\rho_{e}(\tau)(\beta) \,\ \forall \beta \in D\}.
\label{def of GSptau}\end{equation}
\end{definition}
\bigskip

We have:
\begin{equation}
\bigsqcup_{\tau \in \Gal(K_{e}/K)} \, \GIso_{(V, \psi)}^{\tau} \subset \GIso_{(V, \psi)}.
\label{GSp decomposition into GSptau}
\end{equation}
Observe that
\noindent
\begin{equation}
\GIso_{(V, \psi)}^{\id}  = C_{D}( \GIso_{(V, \psi)}).
\label{GSpId is CDGSp}
\end{equation}

\begin{remark}
The bottom horizontal arrow in the diagram (\ref{surjection of motivic Galois groups}) is
\begin{equation}
G_{\mathcal{M}_{K} (M)} \rightarrow G_{\mathcal{M}_{K} (D)} \cong \Gal(K_{e}/K).
\label{map from GMKA to Gal}\end{equation}
Let $g \in G_{\mathcal{M}_{K} (M)}$ and let $\tau := \tau (g)$ be the image of $g$ via 
the map (\ref{map from GMKA to Gal}). Hence for any element $\beta \in D$ considered as an 
endomorphism of $V$ we have:
\begin{equation}
g \beta g^{-1} = \rho_{e}(\tau)(\beta).
\label{GMKA1 acts on V via Galois}
\end{equation}
\end{remark}

\begin{definition}
For any $\tau \in \Gal(K_{e}/K)$, put
\begin{equation}
G_{\mathcal{M}_{K}(M)}^{\tau} := 
\{g \in G_{\mathcal{M}_{K}(M)}:\,\, g \beta g^{-1} = 
\rho_{e}(\tau)(\beta), \,\ \forall \beta \in D\}.
\label{def of GMKAtau}\end{equation}
\end{definition}
\bigskip

It follows from (\ref{GMKA1 acts on V via Galois}), (\ref{def of GMKAtau}), and the surjectivity of 
(\ref{map from GMKA to Gal}) that
\begin{equation}
G_{\mathcal{M}_{K}(M)} \,\, = \bigsqcup_{\tau \in \Gal(K_{e}/K)} \, G_{\mathcal{M}_{K}(M)}^{\tau}
\label{GMKA decomposition into GMKAtau}
\end{equation}
It is clear from (\ref{GMKA subset of GSpV}) and (\ref{def of GSptau}) that 
\begin{equation}
G_{\mathcal{M}_{K}(M)}^{\tau} \subset \GIso_{(V, \psi)}^{\tau}. 
\label{GMKAtau subset of GSptau}\end{equation}
Hence (\ref{GMKA1 acts on V via Galois}) and (\ref{GMKA decomposition into GMKAtau}) give
\begin{equation}
(G_{\mathcal{M}(M)})^{\circ}  \triangleleft G_{\mathcal{M}_{K}(M)}^{{\id}}  \triangleleft  G_{\mathcal{M}_{K}(M)}.
\label{GMKAId normal div in GMKA}
\end{equation}
The map (\ref{map from GMKA to Gal}) gives the following natural map: 
\begin{equation}
G_{\mathcal{M}_{K} (M), 1} \rightarrow  \Gal(K_{e}/K).
\label{map from GMKA1 to Gal}\end{equation}

\begin{definition}
For any $\tau \in \Gal(K_{e}/K)$ put
\begin{equation}
G_{\mathcal{M}_{K}(M), 1}^{\tau} := 
\{g \in G_{\mathcal{M}_{K}(M), 1}:\,\, g \beta g^{-1} = 
\rho_{e}(\tau)(\beta), \,\ \forall \beta \in D\}.
\label{def of GMKA1tau}\end{equation}
\label{Definition of GMKA1tau}\end{definition}
\medskip

\noindent
It follows that there is the following equality
\begin{equation}
G_{\mathcal{M}_{K}(M), 1}^{\tau} = G_{\mathcal{M}_{K}(M), 1} \cap G_{\mathcal{M}_{K}(M)}^{\tau}.
\label{GMKA1tau equals GMKA1 cap GMKAtau}
\end{equation}

\noindent
Let $\tau \in \Gal(K_{e}/K).$ By (\ref{decomposable twisted Lefschetz for fixed element}), 
(\ref{decomposition into twisted Lefschetz for fixed elements}), 
(\ref{GMKAtau subset of GSptau}) 
 we have
\begin{align}
G_{\mathcal{M}_{K}(M), 1}^{\tau} &\subset DL_{K}^{\tau}(V,\, \psi, D)
\label{GMKAtau subset DLKAtau} \\
G_{\mathcal{M}_{K}(M), 1} &\subset DL_{K}(V,\, \psi, D). 
\label{GMKA1 subset of DLKA}
\end{align}
The equality (\ref{GMKA decomposition into GMKAtau}) gives:
\begin{equation}
G_{\mathcal{M}_{K}(M), 1} = \bigsqcup_{\tau \in \Gal(K_{e}/K)} \, G_{\mathcal{M}_{K}(M), 1}^{\tau}
\label{GMKA1 decomposition into GMKA1tau}
\end{equation}
Hence:
\begin{equation}
(G_{\mathcal{M}_K(M), 1})^{\circ}  \,\, \triangleleft \,\, G_{\mathcal{M}_{K}(M), 1}^{{\id}}
\,\, \triangleleft \,\, G_{\mathcal{M}_{K}(M), 1},
\label{GMKA1Id normal div in GMKA1}
\end{equation}
so (\ref{GMKA1tau equals GMKA1 cap GMKAtau}) gives:
\begin{equation}
G_{\mathcal{M}_K(M), 1}  /\, G_{\mathcal{M}_{K}(M), 1}^{{\id}} \,\, \subset \,\, 
G_{\mathcal{M}_K(M)}  / \, G_{\mathcal{M}_{K}(M)}^{{\id}}.   
\label{GMKA1 mod GMKA1Id subset GMKA mod GMKAId}
\end{equation}

\begin{remark}
The $l$-adic representation
\begin{equation}
\rho_l \, : \, G_K \rightarrow \GL(V_l)
\label{l-adic representation associated with M}\end{equation} 
associated with $M$ factors through $G_{\mathcal{M}_{K}(M)} (\Q_l)$
(see \cite[Corollary p.\ 473--474]{Pan} cf. \cite[p.\ 386]{Se94}).
Hence
\begin{equation}
G_{l, K}^{\alg} \subset {G_{\mathcal{M}_{K}(M)}}_{\Q_l} 
\label{GlKalg subset GMKAQl}\end{equation}
where ${G_{\mathcal{M}_{K}(M)}}_{\Q_l} := {G_{\mathcal{M}_{K}(M)}} \otimes_{\Q} \Q_l.$


\section{Motivic Mumford-Tate and Motivic Serre groups}

Since $X/K$ is smooth projective and hence proper, Remarks    
\ref{properties of Hodge-Tate representations } and 
\ref{Hodge-Tate representations in etale cohomology} show that 
$V_l := H^r (\overline{X}, \, \Q_l),$ the $l$-adic realization of 
the motive $h^r(X),$ is of Hodge-Tate type. Hence the image of 
the representation $\rho_l,$
contains an open subset of homotheties of the group $\GL (V_l)$ \cite[Prop.\ 2.8]{Su},
and similarly for any Tate twist such that
$H^r (\overline{X}, \, \Q_l (m))$ has nonzero weights.

\begin{remark}
In the previous statement, the assumption of nonzero weights is essential. 
Indeed, if $X$ has dimension 
$d$, then $H^{2d} (\overline{X}, \, \Q_l (d)) \cong \Q_l$ as $G_K$-modules.
Hence the action of $G_K$ on $H^{2d} (\overline{X}, \, \Q_l (d))$ is trivial,
so the image of the Galois representation is a trivial group 
and hence does not contain homotheties.
\end{remark}

\noindent
From now until the end of the paper, let $M \in \mathcal{M}_{K}$ be a motive which is a direct 
summand of a motive of the form $h^r (X)(m).$ We assume that the $l$-adic realization of $h^r (X)(m)$
has nonzero weights with repect to the $G_K$-action. The $l$-adic realization of $\overline{M}$ is a 
$\Q_l[G_F]$-direct summand of the $l$-adic realization of $h^r (X)(m).$ Hence the $l$-adic 
representation corresponding to $V_l := H_l (\overline{M})$ has image that contains 
an open subgroup of homotheties. 
\medskip

In the following commutative diagram, all horizontal arrows 
are closed immersions and the columns are exact.
$$
\xymatrix{
\quad 1 \ar@<0.1ex>[d]^{}   \quad & \quad \, 1 \ar@<0.1ex>[d]^{}  
\quad & \,\,  1 \ar@<0.1ex>[d]^{}\\
\quad G_{l, K, 1}^{\alg} \ar@<0.1ex>[d]^{} \ar[r]^{}   \quad & 
\quad {G_{\mathcal{M}_{K}(M), 1}}_{\Q_l}  \ar@<0.1ex>[d]^{} \ar[r]^{} \quad & \quad
\Iso_{(V_l, \psi_l)} \ar@<0.1ex>[d]^{}\\ 
\quad G_{l, K}^{\alg} \ar@<0.1ex>[d]^{} \ar[r]^{}   \quad & 
\quad {G_{\mathcal{M}_{K}(M)}}_{\Q_l} \ar@<0.1ex>[d]^{} \ar[r]^{} \quad & \quad
\GIso_{(V_l, \psi_l)} \ar@<0.1ex>[d]^{}\\
\quad\G_{m} \ar@<0.1ex>[d]^{} \ar[r]^{=}  \quad & \quad  
\G_{m} \ar@<0.1ex>[d]^{} \ar[r]^{=}  \quad & \quad  \G_{m} \ar@<0.1ex>[d]^{}\\
\quad 1    \quad & \quad \, 1   \quad & \,\,  1\\}
\label{GlK1 subset GMKA1Ql}$$ 
In particular it follows that:
\begin{equation}
G_{l, K, 1}^{\alg} \subset (G_{\mathcal{M}_{K}(M), 1})_{\Q_l}. 
\label{GlK1alg subset GMKA1Ql}\end{equation}
\end{remark}

We have the following analogue of Theorem~\ref{equality of conn comp 
for Glalg and Glalg1}. 
\begin{theorem}\label{equality of conn comp for GM and GM1} 
Assume that $G_{\mathcal{M}_{K}(M), 1}^{\circ}$ is connected. Then the following map is an isomorphism:
$$i_{M} \,\, : \,\,  \pi_{0} (G_{\mathcal{M}_{K}(M), 1}) \,\,\,\, 
{\stackrel{\cong}{\longrightarrow}} \,\,\,\, \pi_{0} (G_{\mathcal{M}_{K}(M)}).$$  
\end{theorem}
\begin{proof}  
We will write $\mathcal{M}(M)$ for $\mathcal{M}_{K} (M)$ in the following commutative diagram to make notation simpler. 
$$
\xymatrix{
\quad & \quad 1 \ar@<0.1ex>[d]^{}  \quad & \quad 1 \ar@<0.1ex>[d]^{} \quad & 
1 \ar@<0.1ex>[d]^{} \\ 
1 \ar[r]^{} \quad & \quad  (G_{\mathcal{M}(M), 1})^{\circ}\ar@<0.1ex>[d]^{} \ar[r]^{}  \quad 
& \quad G_{\mathcal{M}(M), 1} \ar@<0.1ex>[d]^{} \ar[r]^{}  \quad 
& \quad \pi_{0} (G_{\mathcal{M}(M), 1})   
\ar@<0.1ex>[d]^{i_{M}} \ar[r]^{}  \quad 
& \,\, 1 \\ 
1 \ar[r]^{} \quad & \quad (G_{\mathcal{M}(M)})^{\circ} \ar@<0.1ex>[d]^{} \ar[r]^{}  \quad 
& \quad  G_{\mathcal{M}(M)} \ar@<0.1ex>[d]^{} \ar[r]^{}  \quad 
& \quad \pi_{0}(G_{\mathcal{M}(M)})   \ar@<0.1ex>[d]^{} \ar[r]^{}  \quad 
&  \,\, 1 \\
1 \ar[r]^{} \quad & \quad \G_m  \ar@<0.1ex>[d]^{}  \ar[r]^{=}  
\quad & \quad \G_m  \ar@<0.1ex>[d]^{}  \ar[r]^{} \quad & 1\\
 \quad & \quad 1   \quad & 1 \\}
\label{motivic diagram in Serre theorem}$$

By definition the rows are exact. The middle column is exact by the definition
of $G_{\mathcal{M}_{K}(M), 1}$ and the exactness of the middle column 
in the previous diagram. 
Hence the map $i_M$ is surjective. Since $G_{\mathcal{M}_{K}(M), 1}^{\circ}$ has the 
same dimension as $G_{\mathcal{M}_{K}(M), 1}$ and by assumption 
$G_{\mathcal{M}_{K}(M), 1}^{\circ}$ is connected,
we then have $G_{\mathcal{M}_{K}(M), 1}^{\circ} = (G_{\mathcal{M}_{K}(M), 1})^{\circ}.$ 
Hence the left column is also exact. This shows that $i_M$ is an isomorphism.
\end{proof}

\begin{remark}
Since $G_{\mathcal{M}_K(M)}$ is reductive, the middle vertical column of the diagram 
of the proof of Theorem~\ref{equality of conn comp for GM and GM1}
shows that $G_{\mathcal{M}_K(M), 1}$ is also reductive.
\label{GMKA1 is reductive}\end{remark}

\begin{corollary}\label{equality of quotients concerning for GM and GM1} 
Assume that $G_{\mathcal{M}_{K}(M), 1}^{\circ}$ is connected. Then there are natural isomorphisms

\begin{align}
G_{\mathcal{M}_{K}(M), 1}/\,  G_{\mathcal{M}_{K}(M), 1}^{\id} \,\, 
&{\stackrel{\cong}{\longrightarrow}} \,\, 
G_{\mathcal{M}_{K}(M)} /\, G_{\mathcal{M}_{K}(M)}^{\id},
\label{GMKA1 mod GMKA1Id equals GMKA mod GMKAId} \\
G_{\mathcal{M}_{K}(M), 1}^{\id}/\, (G_{\mathcal{M}_{K}(M), 1})^{\circ} \,\,
&{\stackrel{\cong}{\longrightarrow}} \,\,
G_{\mathcal{M}_{K}(M)}^{\id}/\, (G_{\mathcal{M}_{K}(M)})^{\circ}, 
\label{GMKA1Id mod GMKA1circ equals GMKAId mod GMKAcirc} \\
G_{\mathcal{M}_{K}(M), 1}/\, G_{\mathcal{M}_{K}(M), 1}^{\id} \,\, 
&{\stackrel{\cong}{\longrightarrow}} \,\, \gpDL_{K} (V, \psi, D) /\, \gpDL_{K}^{\id}(V, \psi, D)  
\,\, {\stackrel{\cong}{\longrightarrow}} \,\, \Gal(K_e/K).
\label{GMKA1 mod GMKA1Id equals DLKA mod DLKAId equals GLeK} 
\end{align}
In particular the natural map (\ref{map from GMKA1 to Gal}) is surjective.
\end{corollary}
\begin{proof}  This follows from (\ref{DLKA to Galois Le over K}), 
(\ref{GMKAId normal div in GMKA}), 
(\ref{GMKA1Id normal div in GMKA1}), (\ref{GMKA1 mod GMKA1Id subset GMKA mod GMKAId}),
the surjectivity of (\ref{map from GMKA to Gal}) 
and Theorem \ref{equality of conn comp for GM and GM1}.
\end{proof}

\begin{definition}
The algebraic groups:
\begin{align*}
\MMT_K (M) & := G_{\mathcal{M}_{K}(M)} \\
\MS_K (M) & := G_{\mathcal{M}_{K}(M), 1}
\end{align*}
will be called the \emph{motivic Mumford-Tate group} and (as before) the 
\emph{motivic Serre group} for $M$ respectively.
\label{def of motivic mumford Tate and Sato Tate groups}\end{definition} 

\begin{conjecture} (Motivic Mumford-Tate) \label{Motivic Mumford-Tate} 
For any prime number $l$,
\begin{equation}
G_{l, K}^{\alg} = {\MMT_K (M)}_{\Q_l} .
\label{MMT eq}\end{equation}
\end{conjecture}

By the diagram above Theorem \ref{equality of conn comp for GM and GM1}, 
Conjecture \ref{Motivic Mumford-Tate} is equivalent to the following.
\begin{conjecture} (Motivic Sato-Tate) \label{Motivic Sato-Tate} 
For any prime number $l$,
\begin{equation}
G_{l, K, 1}^{\alg} = \MS_{K}(M)_{\Q_l} .
\label{MST eq}\end{equation}
\end{conjecture}

\begin{remark}
Conjecture \ref{Motivic Mumford-Tate} is equivalent to the conjunction of the following  
equalities:
\begin{align}
(G_{l, K}^{\alg})^{\circ} &= ({\MMT_K (M)}_{\Q_l})^{\circ} 
\label{MMT0 eq} \\
\pi_{0} (G_{l, K}^{\alg}) &= \pi_{0} ({\MMT_K (M)}_{\Q_l}).
\label{GMT0C eq}
\end{align}
Similarly, Conjecture \ref{Motivic Sato-Tate} is equivalent to the conjunction of the following equalities:
\begin{align}
(G_{l, K, 1}^{\alg})^{\circ} &= ({\MS_K (M)}_{\Q_l})^{\circ} 
\label{MST0 eq} \\
\pi_{0} (G_{l, K, 1}^{\alg}) &= \pi_{0} ({\MS_K (M)}_{\Q_l}).
\label{MST0C eq}
\end{align}
\end{remark}


\section{The algebraic Sato-Tate group}

As in the previous section,
we work with motives $M$ which are direct summands of motives of the form 
$h^r (X)(m)$;
in this section, we propose a candidate for the algebraic Sato-Tate group 
for such motives. We prove, under the assumption in Definition
\ref{GMKA1 as AST group}, that our candidate for 
algebraic Sato-Tate group is the expected one.  
In particular the assumption of Definition \ref{GMKA1 as AST group} holds
if $M$ is an AHC motive (see Definition \ref{AHC motives} and Remark 
\ref{A Serre conjecture about connected component of motivic Mumford-Tate}).

\begin{remark}
One observes (\cite[Corollary p.\ 473--474]{Pan}, cf.\ \cite[p.\ 379]{Se94}) that 
\begin{equation}
\MT (V, \psi) \subset (G_{\mathcal{M}_{K}(M)})^{\circ}
\label{MT subset of conn comp of MMT}\end{equation}
Hence we get:
\begin{equation}
\gpH (V, \psi) \subset (G_{\mathcal{M}_{K}(M), 1})^{\circ}
\label{H subset of conn comp of MST}\end{equation}

Recall that $C_D (\Iso_{(V, \psi)}) = \gpDL_{K}^{\id} (V, \psi, D). $ 
It follows by (\ref{GSpId is CDGSp}), 
(\ref{GMKAtau subset of GSptau}), (\ref{GMKAId normal div in GMKA}), and (\ref{MT subset of conn comp of MMT}) we get:
\begin{equation}
\MT (V, \psi) \subset (G_{\mathcal{M}_{K}(M)})^{\circ} \subset G_{\mathcal{M}_{K}(M)}^{\id} \subset 
C_D (\GIso_{(V, \psi)}). 
\label{MTA subset GMKA0 subset CDGSp}
\end{equation}

Similarly by (\ref{GMKAtau subset DLKAtau}), (\ref{GMKA1Id normal div in GMKA1}), and 
(\ref{H subset of conn comp of MST}) that:
\begin{equation}
\gpH (V, \psi) \subset  (G_{\mathcal{M}_{K}(M), 1})^{\circ} \subset 
G_{\mathcal{M}_{K}(M), 1}^{\id} \subset C_D (\Iso_{(V, \psi)}). 
\label{HA subset GMKA10 subset CDSp}
\end{equation}

\begin{remark}\label{approximation form pi0GMK and pi0GMK1}
Observe that (\ref{MTA subset GMKA0 subset CDGSp}) gives an approximation for
$\pi_{0}(G_{\mathcal{M}_{K}(M)}^{\id})$
and (\ref{HA subset GMKA10 subset CDSp}) gives an approximation for 
$\pi_{0}(G_{\mathcal{M}_{K}(M), 1}^{\id})$.
\end{remark}

We observe that for $n$ odd the equality
\begin{equation}
\gpH (V, \psi) = C_D (\Iso_{(V, \psi)})
\label{HA = CD Sp for BGK classes}\end{equation}
is equivalent to the following equality:
\begin{equation}
\MT (V, \psi) = C_D (\GIso_{(V, \psi)}).
\label{MTA = CD GSp for BGK classes}\end{equation}
\end{remark}

\begin{definition} A motive $M \in \mathcal{M}_{K}$ will be called an \emph{AHC motive}
if every Hodge cycle on any object of $\mathcal{M}_{K}(M)$ is an absolute Hodge cycle
(cf. \cite[p.\ 29]{D1}, \cite[p.\ 473]{Pan}).
\label{AHC motives}
\end{definition}

\begin{remark} J-P. Serre conjectured \cite[sec.\ 3.4]{Se94} the equality $\MT (V, \psi) = 
\MMT_{K} (M)^{\circ}$. By \cite{DM} the conjecture holds 
for abelian varieties $A/K$ and for AHC motives $M$ (cf. \cite[Corollary p.\ 474]{Pan}).
\label{A Serre conjecture about connected component of motivic Mumford-Tate}
\end{remark}

\begin{remark} In \cite[p.\ 380]{Se94} there are examples of the computation of  
$\MMT_{K} (M) = G_{\mathcal{M}_{K}(M)}.$ 
In \cite[Theorems 7.3, 7.4]{BK}, we compute $\MMT_{K} (M)$ for abelian varieties
of dimension $\leq 3$ and families of abelian varieties of type I, II and III
in the Albert classification. 
\label{Examples of computation of motivic Mumford-Tate}
\end{remark}
\medskip

If Serre's conjecture $\MT (V, \psi) = 
\MMT_{K} (M)^{\circ}$ holds for $M,$  
then by (\ref{GlKalg subset GMKAQl}) 
the containment (\ref{Mumford-Tate Conj and Deligne Theorem analogue}) 
holds: 

\begin{equation}
(G_{l, K}^{\alg})^{\circ} \subset \MT (V, \psi)_{\Q_l}
\label{Assumption analogues to Deligne Theorem 1}
\end{equation}
and for $n$ odd it is equivalent to:
\begin{equation}
(G_{l, K, 1}^{\alg})^{\circ} \subset \gpH (V, \psi)_{\Q_l}.
\label{Assumption analogues to Deligne Theorem 2}
\end{equation}

\noindent
In particular (\ref{Assumption analogues to Deligne Theorem 1}) and 
(\ref{Assumption analogues to Deligne Theorem 2}) hold for AHC motives (cf.\ Remark 
\ref{A Serre conjecture about connected component of motivic Mumford-Tate}).

\begin{remark}
To obtain $G_{l, K}^{\alg}$ as an extension of scalars to $\Q_l$ of an expected algebraic Sato-Tate 
group defined over $\Q,$ the assumption in the following definition is natural in view of 
(\ref{H subset of conn comp of MST}),
(\ref{Assumption analogues to Deligne Theorem 2}),
Theorem \ref{equality of conn comp for GM and GM1} and Remark 
\ref{A Serre conjecture about connected component of motivic Mumford-Tate}. 
\end{remark}

\begin{definition} \label{GMKA1 as AST group}
Assume that $\MT (V, \psi) = \MMT_{K} (M)^{\circ}.$
Then the \emph{algebraic Sato-Tate group} $\AST_K (M)$ is defined as follows:
\begin{equation}
\AST_K (M) := \MS_K (M). 
\label{AST group}
\end{equation}
Every maximal compact subgroup of $\AST_{K} (M)(\C)$ will be called
a \emph{Sato-Tate group} associated with $M$ and denoted $\ST_K(M).$
\end{definition}
\medskip

\noindent
\begin{theorem}\label{AST as expected extension of HA} 
Assume that we have 
$\MT (V, \psi) = \MMT_{K} (M)^{\circ}.$
Then the group $\AST_{K} (M)$ is reductive and: 
\begin{align}
\AST_{K} (M) &\subset \gpDL_{K}(V, \psi, D),
\label{ASTKA subset DLKA} \\
\AST_{K} (M)^{\circ} &= \gpH(V, \psi) 
\label{(AST)0} \\
\pi_{0} (\AST_{K} (M)) &= \pi_{0} ( \MMT_K (M)),
\label{pi0 AST} \\
\pi_{0} (\AST_{K} (M)) &= \pi_{0} (\ST_{K} (M)).
\label{pi0 AST is pi0 ST}\\
G_{l, K, 1}^{\alg}  &\subset \AST_{K}(M)_{\Q_l}, \,\,\, i.e. \,\,Conjecture \,\,\, \ref{general algebraic Sato Tate conj.} \,
{\rm{(a)}} \,\, holds \,\,  for \,\, M.
\label{Conjecture: general algebraic Sato Tate conj. for M} 
\end{align}
\end{theorem}

\begin{proof}
The group $\AST_{K} (M)$ is reductive by Remark \ref{GMKA1 is reductive}. Moreover  
(\ref{ASTKA subset DLKA}) is just (\ref{GMKA1 subset of DLKA}). 
By assumption $\MT (V, \psi) = (G_{\mathcal{M}_{K}(M)})^{\circ}$ and the equality
$\gpDH(V, \psi) = \gpH(V, \psi)$ (which holds $n$ odd), we have:
\begin{equation}
G_{\mathcal{M}_{K}(M), 1}^{\circ} = (G_{\mathcal{M}_{K}(M)})^{\circ} \cap \Iso_{(V, \psi)} =
\MT (V, \psi) \cap \Iso_{(V, \psi)} = \gpH(V, \psi).
\label{GMKA1 is Hodge under natural assumptions}
\end{equation}
Hence $G_{\mathcal{M}_{K}(M), 1}^{\circ}$ is connected and
\begin{equation}
\AST_{K} (M)^{\circ} = (G_{\mathcal{M}_{K}(M), 1})^{\circ} = G_{\mathcal{M}_{K}(M), 1}^{\circ},
\label{AST GMKAcirc 1 GMKA1 circ}
\end{equation}
so (\ref{(AST)0}) follows.
The equality (\ref{pi0 AST}) follows directly from the Theorem \ref{equality of conn comp for GM and GM1}.  
Equality (\ref{pi0 AST is pi0 ST}) follows since $\AST_{K} (M)^{\circ} (\C)$ is a connected 
complex Lie group and any maximal compact subgroup
of a connected complex Lie group is a connected real Lie group. 
(\ref{Conjecture: general algebraic Sato Tate conj. for M}) follows by (\ref{GlK1alg subset GMKA1Ql}) 
and the assumption (see also Definitions \ref{def of motivic mumford Tate and Sato Tate groups} and 
\ref{GMKA1 as AST group}).
\end{proof}

\begin{corollary}
Under the assumptions that $\MT (V, \psi) = \MMT_{K} (M)^{\circ}$ and  $\gpDH(V, \psi) = \gpH(V, \psi)$,
there are the following commutative diagrams with exact rows:

\begin{equation}
{\xymatrix@C=12pt{
\,\, 0 \ar[r]^{} \,\, & \,\, \gpH(V, \psi) \ar@<0.1ex>[d]^{} \ar[r]^{}   \,\, & 
\,\, \AST_{K} (M)  \ar@<0.1ex>[d]^{}  \ar[r]^{} \,\, & \,\, \pi_{0}( \AST_{K} (M)) 
\ar@<0.1ex>[d]^{}  \ar[r]^{}   \,\, & \,\, 0\\ 
\,\,  0 \ar[r]^{} \,\, & \,\, \gpL(V, \psi, D)  \ar[r]^{}   \,\, & 
\,\, \gpDL_{K} (V, \psi, D) \ar[r]^{}   \,\, & \,\, \pi_{0} ( \gpDL_{K} (V, \psi, D) ) \ar[r]^{}   \,\, & 
\,\, 0\\
}}\label{diagram competibility of AST GMKA1 with DLKA 1} 
\end{equation}
\begin{equation}
{\xymatrix@C=12pt{
\,\, 0 \ar[r]^{} \,\, & \,\, \pi_{0} (G_{\mathcal{M}_{K}(M), 1}^{\id}) \ar@<0.1ex>[d]^{} \ar[r]^{}   \,\, & 
\,\, \pi_{0} ( \AST_{K} (M) ) \ar@<0.1ex>[d]^{}  \ar[r]^{} \,\, & \,\, \Gal(K_e / K)  
\ar@<0.1ex>[d]^{=}  \ar[r]^{}   \,\, & \,\, 0\\ 
\,\,  0 \ar[r]^{} \,\, & \,\, \pi_{0} (\gpDL_{K}^{\id} (V, \psi, D))  \ar[r]^{}   \,\, & 
\,\, \pi_{0}( \gpDL_{K} (V, \psi, D)) \ar[r]^{}   \,\, & \,\, \Gal(K_e / K)  \ar[r]^{}   \,\, & 
\,\, 0\\
}}\label{diagram competibility of AST GMKA1 with DLKA 2}
\end{equation}
\label{commutative diagrams comparing AST with DL}\end{corollary}
\begin{proof}
The exactness of the top row of the Diagram (\ref{diagram competibility of AST GMKA1 with DLKA 1})
follows from (\ref{(AST)0}). The exactness of the top row of the Diagram 
(\ref{diagram competibility of AST GMKA1 with DLKA 2}) follows immediately from
Corollary \ref{equality of quotients concerning for GM and GM1}. 
\end{proof}

\begin{corollary}\label{The natural candidate for AST group example 1}
Assume that $\gpH(V, \psi) = C_{D} (\Iso_{(V, \psi)}).$ Then
\begin{equation}
\AST_{K}(M) = \gpDL_{K} (V, \psi, D).
\end{equation}
\end{corollary}
\begin{proof} It follows by the assumption and 
(\ref{HA subset GMKA10 subset CDSp}) that 
$$\pi_{0}( G_{\mathcal{M}_{K}(M), 1}^{\id}) = \pi_{0}(\gpDL_{K}^{\id} (V, \psi, D)) = 1.$$
Hence the middle vertical arrow in the diagram 
(\ref{diagram competibility of AST GMKA1 with DLKA 2}), which is the right
vertical arrow in the diagram 
(\ref{diagram competibility of AST GMKA1 with DLKA 1}), is an isomorphism.
Since $\gpL(V, \psi, D) = (C_{D} \Iso_{(V, \psi)})^{\circ}$, by assumption we have
$\gpH(V, \psi) = \gpL(V, \psi, D)$. Hence the left 
vertical arrow in the diagram (\ref{diagram competibility of AST GMKA1 with DLKA 1}) is 
an isomorphism, and so the middle vertical arrow in the diagram 
(\ref{diagram competibility of AST GMKA1 with DLKA 1}) is an isomorphism. 
\end{proof}

\begin{corollary} \label{cor alg Sato-Tate for MT explained by endo}
If $\gpH(V, \psi) = C_{D} (\Iso_{(V, \psi)})$ and the Mumford-Tate conjecture holds for $M$,
then the algebraic Sato-Tate conjecture holds:
\begin{equation}
G_{l, K, 1}^{\alg}  = \AST_{K}(M)_{\Q_l}.
\nonumber
\end{equation}
\end{corollary}
\begin{proof} 
By (\ref{GlK1alg subset GMKA1Ql}) and Corollary \ref{The natural candidate for AST group example 1}: 
\begin{equation}
G_{l, K, 1}^{\alg} \subset \AST_{K}(M)_{\Q_l} = \gpDL_{K} (V, \psi, D)_{\Q_l}.
\nonumber\end{equation}
By the assumption $\gpH(V, \psi) = \gpDL_{K_e} (V, \psi, D)$. By virtue of the equivalence of (\ref{ItIsSatoTateGroupIdentification Le}) and
(\ref{ItIsSatoTateGroupIdentification}), we only need to prove that
$(G_{l, K, 1}^{\alg})^{\circ} = \gpH(V, \psi)_{\Q_l}$ which is equivalent to the 
Mumford-Tate conjecture by Remark \ref{Mumford-Tate Conj and Deligne Theorem analogue} . 
\end{proof}

\begin{remark} 
Theorem 
\ref{AST as expected extension of HA} and its Corollaries \ref{commutative diagrams comparing AST with DL}, 
\ref{The natural candidate for AST group example 1} and
\ref{cor alg Sato-Tate for MT explained by endo} show that $\AST_{K} (M)$ from Definition \ref{GMKA1 as AST group} 
is a natural candidate for the algebraic Sato-Tate group for the motive $M$. 
\end{remark}

\begin{remark} Let $M$ be a homogeneous motive which is a direct summand of $h^i(X)(m).$
Put $W := H^i(X(\C),\, \Q (m)).$ If $\psi$ is the polarization of the Hodge stucture on $W$
then we will also denote by $\psi$ the induced polarization on $V = H_{B} (M)$ 
(see Chapter 9).
Observe also that $W_l := H^i_{et}({\overline X},\, \Q_l (m)).$ We will denote by $\rho_{W_l}$ 
the natural representation $\rho_{W_l} \,:\, G_{K} \rightarrow \GIso_{(W_l, \psi_l)}(\Q_l).$  
\label{notation for the theorem: ST with resp. to STK....}\end{remark}

\begin{theorem} Let $M$ be a motive that is a summand of $h^i(X)(m)$ with nonzero weights. 
Let the Hodge structure associated with $M$ have pure odd weight $n.$ Assume that 
Conjecture \ref{general algebraic Sato Tate conj.} {\rm{(a)}} holds for $M$ and there is $c \in \N$ such that $(\Z_{l}^{\times})^c \,\, {\rm{Id}}_{W_l}  \, \subset \, \rho_{W_l} (G_K)$
for all $l.$ Moreover assume that for some $l$ coprime to $c:$
\begin{itemize}
\item[(1)] $K_0 \, \cap \, K(\mu_{\bar{l}}^{\otimes \, n}) \, = \, K,$
\item[(2)] $\gpast_{l, K}$ is an isomorphism with respect to $\rho_l.$
\end{itemize}
Then the Sato-Tate Conjecture holds for the representation $\rho_l \,:\, G_K \rightarrow \GIso_{(V_l, \psi_l)}(\Q_l)$ with respect to $\ST_{K}(M)$ if and only if  it holds for $\rho_l \,:\, G_{K_0} \rightarrow \GIso_{(V_l, \psi_l)}(\Q_l)$ with respect to
$\ST_{K_0}(M).$
\label{ST with resp. to STK the same as with resp. to STK0}
\end{theorem}   

\begin{proof} Because $V_l$ is a subquotient of $W_l$ as a $\Q_l[G_K]$-module, we have  $(\Z_{l}^{\times})^c \,\, {\rm{Id}}_{V_l}  \, \subset \, \rho_{l} (G_K)$ for all $l.$ Since $l$ is coprime to $c$ then $1 + l\, \Z_l \subset (\Z_{l}^{\times})^c.$  
Hence the assumptions in this theorem guarantee that all assumptions of Theorem \ref{STK iff STK0} are satisfied. Hence 
Theorem \ref{ST with resp. to STK the same as with resp. to STK0} follows by Theorem \ref{STK iff STK0}.
\end{proof}



\end{document}